\documentclass[11pt,a4paper,reqno]{amsart}
\usepackage[margin=1.0in,tmargin=0.9in,bmargin=0.80in]{geometry}
\usepackage[dvipsnames]{xcolor}
\usepackage[english]{babel} 

\definecolor{bred}{rgb}{0.8,0,0}

\usepackage{subfigure}
\usepackage{caption}
\usepackage{amsmath,mathrsfs,amssymb,amsfonts,graphicx,epsfig,bbm}
\usepackage{mathtools}
\usepackage{enumerate,amsthm,epstopdf}
\usepackage{enumitem}
\setlist[enumerate]{label={\upshape(\roman*)}}
\usepackage[flushleft]{threeparttable}
\usepackage{multirow}
\usepackage{longtable}
\usepackage[T1]{fontenc}    
\usepackage{bbm}
\usepackage{xargs}
\usepackage{latexsym}
\usepackage{wasysym}
\usepackage{float}
\usepackage{hyperref}       
\usepackage{url}            
\usepackage{booktabs}       
\usepackage{amsfonts}       
\usepackage{nicefrac}       
\usepackage{microtype}      
\usepackage{cleveref}
\usepackage{algorithm,algorithmic}
\usepackage[textsize=footnotesize]{todonotes}
\hypersetup{colorlinks,linkcolor={blue},citecolor={bred},urlcolor={blue}}
\usepackage{natbib}

\newcommand{\bZ}{{\bf Z}}

\newcommand{\bI}{\mbox{\bf I}}

\def\Ebb{\mathbb{E}}
\def \mx {\mathbf{x}}
\def \my {\mathbf{y}}
\def \mv {\mathbf{v}}
\def \mz {\mathbf{z}}

\newtheorem{theorem}{Theorem}[section]
\newtheorem{proposition}[theorem]{Proposition}
\newtheorem{lemma}[theorem]{Lemma}

\newtheorem{remark}[theorem]{Remark}

\allowdisplaybreaks

\begin{document}

\title[Global optimization via optimal transport]{Global optimization via Schr{\"o}dinger-F{\"o}llmer diffusion}

\author[Y. Dai]{Yin Dai}
\author[Y.L. Jiao]{Yuling Jiao}
\author[L.C. Kang]{Lican Kang}
\author[X.L. Lu]{Xiliang Lu}
\author[Z.J. Yang]{Jerry Zhijian Yang}

\address{School of Mathematics and Statistics, Wuhan University, Wuhan 430072, P.R. China}
\email{yindai@whu.edu.cn}

\address{School of Mathematics and Statistics, and
Hubei Key Laboratory of Computational Science, Wuhan University, Wuhan 430072, P.R. China}
\email{yulingjiaomath@whu.edu.cn}

\address{Center for Quantitative Medicine Duke-NUS Medical School, Singapore}
\email{kanglican@duke-nus.edu.sg}

\address{School of Mathematics and Statistics, and
Hubei Key Laboratory of Computational Science, Wuhan University, Wuhan 430072, P.R. China}
\email{xllv.math@whu.edu.cn}

\address{School of Mathematics and Statistics, and
Hubei Key Laboratory of Computational Science, Wuhan University, Wuhan 430072, P.R. China}
\email{zjyang.math@whu.edu.cn}

\date{}


\begin{abstract}
We study the problem of finding global minimizers of $V(x):\mathbb{R}^d\rightarrow\mathbb{R}$ approximately  via sampling from a probability distribution $\mu_{\sigma}$ with  density $p_{\sigma}(x)=\dfrac{\exp(-V(x)/\sigma)}{\int_{\mathbb R^d}
\exp(-V(y)/\sigma) dy }$  with respect to the Lebesgue measure for  $\sigma \in (0,1]$ small enough.
 We analyze a sampler  based on the Euler-Maruyama  discretization of the Schr{\"o}dinger-F{\"o}llmer diffusion processes with stochastic approximation under appropriate assumptions on the step size $s$ and the potential $V$.
 We prove  that the output of the proposed sampler is an approximate global minimizer of $V(x)$ with high probability at cost of sampling $\mathcal{O}(d^{3})$ standard normal random variables.
 Numerical studies illustrate the effectiveness of the proposed method and its superiority to the Langevin method.
\end{abstract}
\maketitle

\section{Introduction}
In this paper we study a challenging problem of finding the global minimizers of a 
non-convex smooth
function $V: \mathbb R^d \rightarrow \mathbb R$.
Suppose $N:=\{x_1^*,\ldots, x_{\kappa}^*\} \subset B_R$ is the set of the global minima of $V$ with finite cardinality, i.e., 
\begin{align}\label{opt}
x_i^* \in \underset{x}{\mbox{argmin}}~ V(x), \quad \text{for any $i=1, \ldots, \kappa$,}
\end{align}
where $B_R$ denotes the ball  centered at origin  with radius $R>0$.
Precisely speaking,
 we have $ \mathscr{L} (V(x)<V(x_i^*)-\varepsilon)=0$ and $\mathscr{L} (V(x)<V(x_i^*)+ \varepsilon)>0$ for any $\varepsilon >0$, where   $\mathscr{L}$ denotes the Lebesgue measure on $\mathbb R^d$.
 Without loss of generality, we can assume $V(x)=\|x \|_2^2/2$ 
 outside the ball $B_R$
 without changing the global minimizers of $V$.
For any given $\sigma \in (0,1]$, we define a constant $C_{\sigma}$ and a probability density function $p_{\sigma}(x)$ on $\mathbb R^d$ as
$$p_{\sigma}(x) :=\frac{\exp(-V(x)/\sigma)}{C_{\sigma}}, \quad \mathrm{for} \; C_{\sigma} :=\int_{\mathbb R^d} \exp(-V(x)/\sigma)dx <\infty .$$
Let $\mu_{\sigma}$ be the probability distribution corresponding to the density function $p_{\sigma}$. If the function $V$ is twice continuously differentiable, we have that the measure $\mu_{\sigma}$ converges weakly to a probability measure with weights proportional to $\dfrac{\left(\det \nabla^2 V(x_i^*) \right)^{-\frac 1 2}}{\sum_{j=1}^\kappa \left(\det \nabla^2 V(x_j^*) \right) ^{-\frac 1 2}}$ at $x_i^*$ as $\sigma$ goes to 0, i.e.,
$$\underset{\sigma\rightarrow 0}{\lim}~\mu_{\sigma}=\dfrac{\sum_{i=1}^\kappa \left( \det \nabla^2 V(x_i^*) \right)^{-\frac 1 2}\delta_{x_i^*}}{\sum_{j=1}^\kappa \left(\det \nabla^2 V(x_j^*) \right) ^{-\frac 1 2}}.$$
Therefore, solving the optimization problem \eqref{opt} can be converted into sampling from the probability distribution measure $\mu_{\sigma}$ for sufficiently small   $\sigma$.

An efficient method 
sampling from
$\mu_{\sigma}$ is 
based on 
the overdamped Langevin stochastic differential equation (SDE) which is given by
\begin{equation}\label{Langevin}
  \mathrm{d} Z_t= -\nabla V(Z_t) \mathrm{d}t +\sqrt{2\sigma} \mathrm{d} B_t,  
\end{equation}
where $(B_t)_{t\ge 0}$ is a $d$-dimensional Brownian motion. 
Under some certain conditions,
Langevin SDE \eqref{Langevin} admits  the unique invariant measure $\mu_{\sigma}$ \citep{bakry2008rate}. 
Hence the Langevin sampler is generated by 
applying Euler-Maruyama discretization on 
this process to achieve the purpose of sampling from $\mu_{\sigma}$. 
Convergence properties of the Langevin sampler under the strongly convex potential assumption have been established by 
\cite{durmus2016sampling,dalalyan2017further,dalalyan2017theoretical,cheng2018convergence,durmus2016high-dimensional,dalalyan2019user-friendly}.
Moreover, the strongly convex potential assumption can be replaced by
different conditions to guarantee the log-Sobolev inequality for the target distribution,  including the dissipativity condition for the drift term \citep{raginsky2017non,zhang2019nonasymptotic,mou2019improved} and  the local convexity condition for the potential function outside a ball \citep{durmus2017nonasymptotic,ma2019sampling,bou2020coupling}.

An alternative to 
Langevin sampler
is the class of algorithms based on the Schr\"{o}dinger-F\"{o}llmer diffusion process \eqref{sde}. 
This process has been proposed for sampling and generative modelling \citep{tzen2019theoretical,sfsHuang,sfsJiao,wang2021deep}. Subsequently, \cite{ruzayqat2022unbiased} studied the problem of 
unbiased estimation of expectations based on the Schr\"{o}dinger-F\"{o}llmer diffusion process. \cite{vargas2021bayesian} applied this process to Bayesian inference in large datasets, 
and the related posterior is reached in finite time. 
\cite{zhang2021path} proposed a new Path Integral Sampler (PIS), a generic sampler that generates samples through simulating a target-dependent SDE which can be trained with free-form architecture network design. The PIS is built on Sch\"odinger-F\"ollmer diffusion process to reach the terminal distribution. Different from these existing works, the purpose of this paper is to solve the
non-convex smooth optimization problems.
To that end,
we need to rescale the Schr\"{o}dinger-F\"{o}llmer diffusion process to sample from the target distribution $\mu_{\sigma}$.

To be precise, the Schr\"{o}dinger-F\"{o}llmer diffusion process associated to $\mu_{\sigma}$ is defined as
\begin{align}\label{sde}
\mathrm{d} X_{t}= \sigma b\left(X_{t}, t\right) \mathrm{d} t+\sqrt{\sigma} \mathrm{d} B_{t}, \quad t \in[0,1],\quad X_0=0,
\end{align}
where the drift function is $$b(x,t)=\frac{\mathbb E_{Z\sim N\left(0,\bf{I}_d\right)}[\nabla f_{\sigma}(x+\sqrt{(1-t)\sigma}Z)]}{\mathbb E_{Z\sim N\left(0,\bf{I}_d\right)}[ f_{\sigma}(x+\sqrt{(1-t)\sigma} Z)]}:\mathbb{R}^d \times [0,1]\rightarrow \mathbb{R}^d$$
 with the density ratio $f_{\sigma}(\cdot)=\frac{p_{\sigma}(\cdot) }{\phi_{\sigma}(\cdot)}$ and $\phi_{\sigma}(\cdot)$ being the density function of a normal distribution $N(0,\sigma \mathbf{I}_d)$.  According to \cite{leonard2014survey} and \cite{eldan2020}, the process  $\{X_t\}_{t\in [0,1]}$  in \eqref{sde} was first formulated by F\"{o}llmer \citep{follmer1985,follmer1986,follmer1988} when studying the Schr\"{o}dinger bridge problem \citep{schrodinger1932theorie}.
The main feature of the above  Schr\"{o}dinger-F\"{o}llmer  diffusion process is that  it interpolates $\delta_{0}$ and  $\mu_{\sigma}$ in time $[0,1]$, i.e.,  $X_1 \sim \mu_{\sigma}$, see Proposition \ref{SBP}. Then we can solve the optimization problem \eqref{opt} by sampling from $\mu_{\sigma}$ via the following Euler-Maruyama discretization of \eqref{sde},
$$
Y_{t_{k+1}}=Y_{t_k}+ \sigma s b\left(Y_{t_k}, t_k\right)+\sqrt{\sigma s}\epsilon_{k+1},
~
Y_{t_0} = 0,~ k=0,1,\ldots, K-1,
$$
where $s = 1/K$ is the  step size, $t_k =k s$,  and $\{\epsilon_{k}\}_{k=1}^{K}$ are independent and identically distributed from $N(0,\bI_{d})$.
If the expectations in the drift term $b(x,t)$ do not have  analytical forms, one can use Monte Carlo method to evaluate $b\left(Y_{t_k}, t_k\right)$ approximately, i.e., one can sample from $\mu_{\sigma}$ according to
 $$
\widetilde{Y}_{t_{k+1}}=\widetilde{Y}_{t_k}+ \sigma s\tilde{b}_m\left(\widetilde{Y}_{t_k}, t_k\right)+\sqrt{\sigma s}\epsilon_{k+1},
~\widetilde{Y}_{t_0} = 0, ~k=0,1,\ldots, K-1,$$
where $\tilde{b}_m(\widetilde{Y}_{t_{k}},t_{k}):=
\dfrac{\frac{1}{m}\sum_{j=1}^m[\nabla  f_{\sigma}(\widetilde{Y}_{t_{k}}+\sqrt{(1-t_{k})\sigma}Z_j)]}{\frac{1}{m}\sum_{j=1}^{m} [ f_{\sigma}(\widetilde{Y}_{t_{k}}+\sqrt{(1-t_{k})\sigma} Z_j)]}$ with $Z_1,...,Z_m$ i.i.d. $N(0,\bI_d)$.

The main result of this paper is summarized in the following.
\begin{theorem} (Informal)
Under   condition  $(\textbf{A})$, 
$\forall \ \  0 < \delta \ll 1$,  with probability at least $1-\sqrt{\delta}$,  $\widetilde{Y}_{t_K} $ is a $\tau$-global minimizer of $V$, i.e.,
$V(\widetilde{Y}_{t_K})\leq \tau + \inf V(x)$, 
if number of iterations $K\geq \mathcal{O}\left(\frac{d^2}{\delta}\right)$, number of Gaussian samples per iteration  $m \geq \mathcal{O}\left(\frac{d}{\delta}\right)$ and $\sigma \leq \mathcal{O} \left(\frac{\tau}{\log(1/\delta)}\right)$. 
\end{theorem}

The rest of this paper is organized as follows. In Section \ref{method},  we  give the motivation and details of approximation method, i.e., Algorithm \ref{alg:1}.
In Section \ref{Theorey}, we present the theoretical analysis for the proposed method. 
In Section \ref{simulation}, a numerical example is given to validate the efficiency of the method.
We conclude the manuscript in Section \ref{conlusion}. Proofs for all the propositions and theorems are provided in Appendix \ref{append}.

\section{Methodology Description}\label{method}

In this section we first provide some background on the Schr{\"o}dinger-F{\"o}llmer diffusion. Then we propose Algorithm \ref{alg:1} to solve the minimization problem \eqref{opt} based on the Euler-Maruyama discretization scheme of the Schr{\"o}dinger-F{\"o}llmer diffusion.

\subsection{Background on Schr{\"o}dinger-F{\"o}llmer diffusion}
We first recall the Schr{\"o}dinger bridge problem, then introduce the Schr{\"o}dinger-F{\"o}llmer diffusion.

\subsubsection{Schr{\"o}dinger bridge problem}
Let $\Omega = C([0,1],\mathbb{R}^d)$ be the space of $\mathbb{R}^d$-valued continuous functions on the time interval $[0, 1]$. Denote $Z = (Z_t)_{t\in [0,1]}$ as the canonical process on $\Omega$, where $Z_t(\omega) = \omega_t$, $\omega = (\omega_s)_{s\in [0,1]}\in \Omega$.
The canonical $\sigma$-field on $\Omega$ is then generated as $\mathscr{F}  = \sigma(Z_t,t\in[0,1]) = \left\{\{\omega:(Z_t(\omega))_{t\in [0,1]}\in H\}:H\in\mathcal{B}(\mathbb{R}^d)\right\}$, where $\mathcal B(\mathbb R^d)$ denotes the Borel $\sigma$-algebra of $\mathbb R^d$. Denote $\mathcal{P}(\Omega)$ as the space of probability measures on the path space $\Omega$, and $\mathbf{W}^{\mx}_{\sigma} \in\mathcal{P}(\Omega)$ as the Wiener measure with variance $\sigma$ whose initial marginal is $\delta_{\mx}$. The law of the reversible Brownian motion is then defined as $\mathbf{P}_{\sigma}= \int \mathbf{W}_{\sigma}^{\mx}\mathrm{d}\mx$, which is an unbounded measure on $\Omega$. One can observe that $\mathbf{P}_{\sigma}$ has a marginal distribution coinciding with the Lebesgue measure $\mathscr{L}$ at each $t$. \cite{schrodinger1932theorie} studied the problem of finding the most likely random evolution between two probability distributions $\widetilde{\nu}, \widetilde{\mu} \in \mathcal{P}(\mathbb{R}^d)$.
This problem is referred to as the Schr\"{o}dinger bridge problem (SBP).
SBP can be further formulated as seeking a probability law on the path space that interpolates between $\widetilde{\nu}$ and $\widetilde{\mu}$, such that the probability law is close to the prior law of the Brownian diffusion with respect to the relative entropy \citep{jamison1975markov,leonard2014survey}, i.e.,  finding a path measure $\mathbf{Q}^* \in \mathcal{P}(\Omega)$ with marginal $\mathbf{Q}^*_{t} = (Z_t)_{\#}\mathbf{Q}^*=\mathbf{Q}^*\circ Z_t^{-1}, t\in [0,1]$ such that
$$\mathbf{Q}^* \in \arg \min \mathbb{D}_{\mathrm{KL}}(\mathbf{Q}||\mathbf{P}_{\sigma}),
\;\textrm{with}\;\mathbf{Q}_0 = \widetilde{\nu}, \qquad \mathbf{Q}_1 = \widetilde{\mu},$$ 
where the relative entropy is defined by $\mathbb{D}_{\mathrm{KL}}(\mathbf{Q}||\mathbf{P}_{\sigma}) = \int \log\left(\dfrac{d \mathbf{Q}}{d \mathbf{P}_{\sigma}}\right) d \mathbf{Q} $ if $\mathbf{Q}\ll \mathbf{P}_{\sigma}$ (i.e., $\mathbf{Q}$ is absolutely continuous w.r.t. $\mathbf{P}_{\sigma}$), and $\mathbb{D}_{\mathrm{KL}}(\mathbf{Q}||\mathbf{P}_{\sigma}) =+\infty$ otherwise.
The following theorem characterizes the solution of SBP.

\begin{theorem}[\cite{leonard2014survey}]\label{th01}
If measures $\widetilde{\nu}, \widetilde{\mu} \ll \mathscr{L}$,  then SBP admits a unique solution $d\mathbf{Q}^* = f^*(Z_0)g^*(Z_1)d\mathbf{P}_{\sigma}$, where
$f^*$ and $g^*$ are $\mathscr{L}$-measurable non-negative  functions satisfying the  Schr\"{o}dinger system
$$\left\{\begin{array}{l}
f^*(\mx) \mathbb{E}_{\mathbf{P}_{\sigma}}\left[g^*\left(Z_{1}\right) \mid Z_{0}=\mx\right]= \dfrac{d \widetilde{\nu}}{d\mathscr{L}}(\mx), \quad \mathscr{L}\text{-a.e.}, \\
g^*(\my)  \mathbb{E}_{\mathbf{P}_{\sigma}}\left[f^{*}\left(Z_{0}\right) \mid Z_{1}=\my\right]=\dfrac{d \widetilde{\mu}}{d\mathscr{L}}(\my),  \quad \mathscr{L}\text{-a.e.}.
\end{array}\right.$$
Furthermore, the pair $(\mathbf{Q}^*_{t},\mv^*_{t})$ with $$\mv^*_{t}(\mx) = \nabla_{\mx} \log \mathbb{E}_{\mathbf{P}_{\sigma}}\left[g^{*}\left(Z_{1}\right) \mid Z_{t}= \mx\right]$$ solves the minimum action problem
$$\min_{\mu_t,\mv_t} \int_{0}^{1} \mathbb{E}_{\mz\sim \mu_t}[\|\mv_t(\mz)\|^2] \mathrm{d} t$$ such that
$$\left\{\begin{array}{l}
\partial_{t}\mu_t = -\nabla \cdot(\mu_t \mv_t) +  \frac{\sigma \Delta \mu_t}{2}, \quad \text{on } (0,1) \times \mathbb{R}^{d},\\
\mu_{0}=\widetilde{\nu}, \mu_{1}=\widetilde{\mu}.
\end{array}\right.
$$
\end{theorem}
Let  $K_{\sigma}(s, \mx, t, \my) = [2\pi\sigma(t-s)]^{-d/2}\exp\left(-\dfrac{\|\mx - \my\|^2}{2\sigma(t-s)}\right)$ be the transition density of the Wiener process with variance $\sigma$, $\widetilde{q}(\mx)$ and $\widetilde{p}(\my)$ be
the density of $\widetilde{\nu}$ and $\widetilde{\mu}$, respectively.  Denote by
$$f_{0}(\mx) = f^*(\mx), \ \ g_{1}(\my) = g^*(\my),$$
$${f_{1}}(\my) = \mathbb{E}_{\mathbf{P}_{\sigma}}\left[f^{*}\left(Z_{0}\right) \mid Z_{1}=\my\right] = \int K_{\sigma}(0, \mx, 1, \my)f_{0}(\mx) d \mx,$$
$${g_{0}}(\mx)= \mathbb{E}_{\mathbf{P}_{\sigma}}\left[g^*\left(Z_{1}\right) \mid Z_{0}=\mx\right] = \int K_{\sigma}(0, \mx, 1, \my)g_{1}(\my) d \my.$$
Then the  Schr\"{o}dinger system  in Theorem \ref{th01} can also be characterized by
\begin{equation*}
\widetilde{q}(\mx) = f_0(\mx) {g_{0}}(\mx), \ \  \widetilde{p}(\my)=  {f_{1}}(\my)g_1(\my),
\end{equation*}
with the following forward and backward time harmonic equations  \citep{chen2020stochastic},
 $$\left\{\begin{array}{l}
\partial_t f_t(\mx) = \frac{\sigma \Delta}{2} f_t(\mx),  \\
\partial_t g_t(\mx) = -\frac{\sigma \Delta}{2} g_t(\mx),
\end{array}\right. \quad \text { on }(0,1) \times \mathbb{R}^{d}.
$$

Let $q_t$ denote marginal density of $\mathbf{Q}_t^{*}$, i.e., $q_t(\mx) = \frac{d\mathbf{Q}_t^{*}}{d\mathscr{L}}(\mx)$,  then it can be represented by the product of  $g_t$  and $f_t$ \citep{chen2020stochastic}. Let  $\mathcal{V}$ consist of admissible Markov controls with finite energy. Then,  the vector field
\begin{equation}\label{drift}
\begin{aligned}
\mv^*_{t}=\sigma \nabla_{\mx}\log g_t(\mx)
=\sigma \nabla_{\mx}\log  \int K_{\sigma}(t, \mx, 1, \my)g_1(\my) \mathrm{d} \my
\end{aligned}
\end{equation}
solves the following stochastic control problem.
\begin{theorem}[\cite{dai1991stochastic}]\label{th02}
$$\mathbf{v}^*_{t}(\mx)\in \arg\min_{\mathbf{v} \in \mathcal{V}}\mathbb{E}\left[\int_0^1\frac{1}{2 \sigma}\|\mathbf{v}_t\|^2\mathrm{d}t\right]$$
such that
\begin{equation}\label{sdeb}
\left\{\begin{array}{l}
\mathrm{d}\mx_t = \mathbf{v}_t \mathrm{d}t+ \sqrt{\sigma} \mathrm{d} B_t, \\
\mx_0\sim \widetilde{q}(\mx),\quad \mx_1\sim \widetilde{p}(\mx).
\end{array}\right.
\end{equation}
\end{theorem}
According to Theorem \ref{th02}, the dynamics determined by the SDE in \eqref{sdeb} with a time-varying drift term $\mathbf{v}^*_{t}$ in \eqref{drift} will drive the particles sampled from the initial distribution $\widetilde{\nu}$ to evolve to the particles drawn from the target distribution $\widetilde{\mu}$ on the unit time interval. This nice property is what we need in designing samplers: we can sample from the underlying target distribution $\widetilde{\mu}$ via pushing forward a simple reference distribution $\widetilde{\nu}$.
In particular, if we take the initial distribution $\widetilde{\nu}$ to be $\delta_0$, the degenerate distribution at $0$, then the Schr\"{o}dinger-F\"{o}llmer diffusion process \eqref{sch-equation} defined below is a solution to \eqref{sdeb}, i.e., it
will transport $\delta_0$ to the target distribution.

\subsubsection{Schr\"{o}dinger-F\"{o}llmer diffusion process}
From now on, without loss of generality, we can assume that the minimum value of $V$ is 0, i.e., $V(x_i^*)=0,i=1,\ldots,\kappa$, otherwise, we consider $V$ replaced by $V-\min_{x} V(x)$.
Since $\mu_{\sigma}$ is absolutely continuous with respect to the $d$-dimensional Gaussian distribution $N(0,\sigma \bI_d)$, then we denote the Radon-Nikodym derivative of $\mu_{\sigma}$ with respect to $N(0, \sigma \bI_d)$ as follows: 
\begin{equation}\label{drb}
 f_{\sigma}(x) :=\frac{d\mu_{\sigma}}{dN(0,\sigma \bI_d)}(x)=\frac{p_{\sigma}(x)}{\phi_{\sigma}(x)},\ x \in \mathbb{R}^d.
\end{equation}
Let $Q_t$ be the heat semigroup defined by
\begin{equation*}
Q_{t} f_{\sigma}(x):= \Ebb_{Z \sim N(0,\bI_d)}[f_{\sigma}(x+\sqrt{t\sigma} Z)], \quad t \in [0, 1].
\end{equation*}
The Schr\"{o}dinger-F\"{o}llmer diffusion process $\{X_t\}_{t\in [0,1]}$ \citep{follmer1985, follmer1986, follmer1988}
is defined as 
\begin{align}\label{sch-equation}
\mathrm{d} X_{t}=\sigma b(X_{t}, t) \mathrm{d}t+\sqrt{\sigma} \mathrm{d} B_{t},  \ X_{0}=0,  \  t \in [0,1],
\end{align}
where $b(x,t): \mathbb R^d \times [0,1] \rightarrow \mathbb R^d$ is the drift term given by
\begin{equation*}
b(x, t)=\nabla \log Q_{1-t} f_{\sigma}(x).
\end{equation*}
This process $\{X_t\}_{t\in [0,1]}$ defined in \eqref{sch-equation} is a solution to \eqref{sdeb} with
 $\tilde{\nu} = \delta_{0}$, $\tilde{\mu} = \mu_{\sigma}$, and $\mathbf{v}_{t}(x) = b(x,t)$ \citep{dai1991stochastic, lehec2013representation, eldan2020}. Let 
  \begin{equation*}
  \hat f_{\sigma}(x) :=C_{\sigma} (2\pi \sigma)^{-\frac d 2} f_{\sigma}(x).
  \end{equation*}
  Since the drift term $b(x,t)$ is scale-invariant with respect to $f_{\sigma}$ in the sense that
$b(x, t)=\nabla \log Q_{1-t} Cf_{\sigma}(x)$ for any $C>0$. Therefore, the Schr{\"o}dinger-F\"{o}llmer diffusion can be used for sampling from an unnormalized distribution $\mu_{\sigma}$,  that is,  the normalizing constant $C_{\sigma}$ of  $\mu_{\sigma}$ does not need to be known.
We have $\mu_{\sigma}(dx)=\exp (-V(x)/\sigma) d x / C_{\sigma}$ with the normalized constant $C_{\sigma}$, then $f_{\sigma}(x)=\frac{\left(\sqrt{2\pi\sigma}\right)^d}
{C_{\sigma}}\exp (-V(x)/\sigma+\frac{\|x\|_2^2}{2 \sigma})$
and 
$\hat f_{\sigma}(x)=\exp\left(-\frac{V(x)}{\sigma} +\frac{\|x \|^2_2}{2\sigma} \right)$. 

To ensure that the SDE \eqref{sch-equation} admits a unique strong solution, we assume that
 \begin{itemize}
\item[$(\textbf{A})$] 
$V(x)$ is twice continuous differentiable on $
\mathbb{R}^d$  and  $V(x)=\|x \|_2^2/2$ 
 outside a  ball $B_R$.
\end{itemize}
As we  mentioned earlier, we can make  Assumption $\mathbf{(A)}$ by smoothing  and  it does not change the  the global minimizers of $V$.
Under Assumption $\mathbf{(A)}$, we have the following two properties which further imply  the SDE \eqref{sch-equation} admits a unique strong solution.
\begin{itemize}
\item[$(\textbf{P1})$] For each $\sigma \in (0,1]$, $\hat f_{\sigma},\nabla \hat f_{\sigma}$ are Lipschitz continuous with a constant $\gamma_{\sigma} >0$,
where 
\begin{align*}
 \gamma_{\sigma}:&= \max _{\|x\|_2 \leq R}\exp \left(-\frac{V(x)}{\sigma} + \frac{\|x\|_2^{2}}{2\sigma}\right)  \left \lbrace \left\|\frac{x}{\sigma} -\frac{\nabla V(x)}{\sigma} \right\|_2^2+ \left\|\frac{\bI_d}{\sigma}- \frac{\nabla^2 V(x)}{\sigma} \right\|_2 \right\rbrace \\
 &= \left \lbrace \left(M_{2,R}/\sigma \right)^2 + M_{3,R}/ \sigma \right \rbrace \exp\left(M_{1,R}/\sigma \right),
\end{align*}
and
\begin{equation}\label{eqn:M1-M3}
M_{1,R} := \max_{\|x \|_2 \le R} \left\lbrace -V(x)+ \frac{\|x \|^2_2}{2} \right \rbrace,\; M_{2,R} := \max_{\|x \|_2 \le R} \left\| x -\nabla V(x) \right\|_2, \;
M_{3,R} := \max_{\|x \|_2 \le R} \left\| \mathbf{I}_d - \nabla^2 V(x) \right\|_2.
\end{equation}
\item[$(\textbf{P2})$] For each $\sigma \in (0,1]$, there exists a constant $\xi_{\sigma} >0$ such that $\hat f_{\sigma} \geq \xi_{\sigma}$, 
where 
\begin{equation}\label{eqn:m1}
\xi_{\sigma} :=\exp\left(m_{1,R}/\sigma \right)\; \mbox{with} \; 
 m_{1,R} :=\min_{\|x \|_2 \le R} \left\lbrace -V(x)+\frac{\|x \|_2^2}{2} \right \rbrace.
\end{equation}
\end{itemize}
Properties  (\textbf{P1})-(\textbf{P2}) are
 shown   in Section \ref{vc12} in  Appendix. We should mention here  (\textbf{P1})-(\textbf{P2}) are 
used as assumptions  in \cite{lehec2013representation,tzen2019theoretical}.


Thanks to  (\textbf{P1}) and (\textbf{P2}), some calculations show that  $$\|\nabla \hat f_{\sigma} \|_2\leq \gamma_{\sigma}, \quad \|\nabla^2 \hat f_{\sigma} \|_2\leq \gamma_{\sigma},$$ 
$$
\sup_{x\in\mathbb{R}^d,t\in[0,1]}\|\nabla Q_{1-t} \hat f_{\sigma}(x)\|_2\leq \gamma_{\sigma}, \quad \sup_{x\in\mathbb{R}^d,t\in[0,1]}\|\nabla^2(Q_{1-t} \hat f_{\sigma}(x))\|_2\leq \gamma_{\sigma},
$$ and
$$
b(x,t)=\frac{\nabla Q_{1-t} \hat f_{\sigma}(x)}{Q_{1-t} \hat f_{\sigma}(x)}, ~\nabla b(x,t)=\frac{\nabla^2 (Q_{1-t} \hat f_{\sigma})(x)}{Q_{1-t} \hat f_{\sigma}(x)}-b(x,t)b(x,t)^{\top}.
$$
We conclude that
$$
\sup_{x\in\mathbb{R}^d,t\in[0,1]}\|b(x,t)\|_2\leq \gamma_{\sigma}/\xi_{\sigma},
\sup_{x\in\mathbb{R}^d,t\in[0,1]}\|\nabla b(x,t)\|_2\leq \gamma_{\sigma} /\xi_{\sigma} +\gamma^2_{\sigma}/\xi_{\sigma}^2.
$$
Furthermore, we can also easily deduce that the drift term $b$ satisfies a linear growth condition and a Lipschitz continuity condition \citep{revuz2013continuous,pavliotis2014stochastic}, that is,
\begin{align}\label{cond1}
\|b(x,t)\|_2^2\leq C_{0,\sigma} (1+\|x\|_2^2) \tag{C1}, \ x \in \mathbb{R}^d, t \in [0,1]
\end{align}
and
\begin{align}\label{cond2}
\|b(x,t)-b(y,t)\|_2\leq C_{1,\sigma} \|x-y\|_2 \tag{C2}, \ x, y \in \mathbb{R}^d, t \in [0,1],
\end{align}
where $C_{0,\sigma}$ and $C_{1,\sigma}$ are two finite positive constants  that only depend on $\sigma$.
The linear growth condition \eqref{cond1} and  Lipschitz continuity condition \eqref{cond2} ensure the existence of the unique strong solution of
Schr{\"o}dinger-F\"{o}llmer SDE  \eqref{sch-equation}.
We summarize the above discussion in the following Proposition, whose proof are shown in Appendix \ref{append}. 

\begin{proposition}\label{SBP}
Under  assumption $(\textbf{A})$
the Schr{\"o}dinger-F\"{o}llmer
SDE  \eqref{sch-equation}
has a unique strong solution $\{X_t\}_{t\in[0,1]}$ with $X_0 \sim \delta_0$ and  $X_{1} \sim \mu_{\sigma}$.
\end{proposition}


\subsection{Euler-Maruyama discretization for Schr{\"o}dinger-F{\"o}llmer Diffusion}\label{discretization}
Proposition   \ref{SBP} shows  that the Schr{\"o}dinger-F\"{o}llmer diffusion will transport $\delta_0$ to the probability distribution measure $\mu_{\sigma}$ on the unite time interval. Because the drift term $b(x,t)$ is scale-invariant with respect to $f_{\sigma}$ in the sense that $b(x, t)=\nabla \log Q_{1-t} Cf_{\sigma}(x), \forall C>0$, the Schr{\"o}dinger-F\"{o}llmer diffusion can be used for sampling from $\mu_{\sigma}(dx)=\exp (-V(x)/\sigma) d x/C_{\sigma}$,  where  the normalizing constant  $C_{\sigma}$ may not be known. To this end, we use the  Euler-Maruyama method to discretize  the  Schr{\"o}dinger-F\"{o}llmer diffusion \eqref{sch-equation}.
Let $$t_{k}=k\cdot s, \ \ ~k=0,1,  \ldots, K, \ \ ~\mbox{with}~ \ \ s=1 / K, \ \ Y_{t_{0}}=0,$$
the Euler-Maruyama discretization scheme reads
\begin{align}\label{emd}
Y_{t_{k+1}}=Y_{t_k}+\sigma s b\left(Y_{t_{k}}, t_{k}\right)+\sqrt{\sigma s}\epsilon_{k+1},
~
k=0,1,\ldots, K-1,
\end{align}
where   $\{\epsilon_{k}\}_{k=1}^{K}$ are i.i.d. $N(0,\bI_d)$ and
{
\begin{align}\label{driftb}
b(Y_{t_{k}},t_{k})=\frac{\Ebb_{Z}[\nabla \hat f_{\sigma}(Y_{t_{k}}+\sqrt{(1-t_{k})\sigma}Z)]}{\Ebb_{Z}[\hat f_{\sigma}(Y_{t_{k}}+\sqrt{(1-t_{k})\sigma} Z)]}=\frac{\Ebb_{Z}[Z \hat f_{\sigma}(Y_{t_{k}}+\sqrt{(1-t_{k})\sigma}Z)]}{\Ebb_{Z}[\hat f_{\sigma}(Y_{t_{k}}+\sqrt{(1-t_{k})\sigma} Z)]\sqrt{(1-t_k)\sigma}},
\end{align}
}
where the second equality follows from Stein's lemma
\citep{stein1972, stein1986,landsman2008stein}.
From the definition of $b(Y_{t_{k}},t_{k})$  in \eqref{driftb}  we may not get its explicit expression.
We then consider a estimator $\tilde{b}_{m}$ of $b$ by replacing $\Ebb_{Z}$ in the drift term $b$ with $m$-samples mean, i.e.,
\begin{align}\label{drifte1}
\tilde{b}_m(Y_{t_{k}},t_{k})=
\frac{\frac{1}{m}\sum_{j=1}^m[\nabla \hat f_{\sigma}(Y_{t_{k}}+\sqrt{(1-t_{k})\sigma}Z_j)]}{\frac{1}{m}\sum_{j=1}^{m} [\hat f_{\sigma}(Y_{t_{k}}+\sqrt{(1-t_{k})\sigma} Z_j)]}, \ k=0, \ldots, K-1,
\end{align}
or
\begin{align}\label{drifte2}
\tilde{b}_m(Y_{t_{k}},t_{k})=
\frac{\frac{1}{m}\sum_{j=1}^m[Z_j \hat f_{\sigma}(Y_{t_{k}}+\sqrt{(1-t_{k})\sigma}Z_j)]}{\frac{1}{m}\sum_{j=1}^{m} [\hat f_{\sigma}(Y_{t_{k}}+\sqrt{(1-t_{k})\sigma} Z_j)]\cdot\sqrt{(1-t_k)\sigma} }, \ k=0, \ldots, K-1,
\end{align}
where $Z_1, \ldots, Z_{m}$ are i.i.d. $N(0, \bI_d)$. The detailed description of the proposed method is summarized in the following Algorithm \ref{alg:1} below.
\begin{algorithm}[H]
	\caption{Solving \eqref{opt} via Euler-Maruyama discretization of  Schr{\"o}dinger-F{\"o}llmer Diffusion}
    \label{alg:1}
	\begin{algorithmic}[1]
\STATE Input: $\sigma$,  $m$, $K$.  Initialize $s=1/K$, $\widetilde{Y}_{t_0}=0$.
\FOR{$k= 0,1,\ldots, K-1$ }
\STATE Sample $\epsilon_{k+1}\sim N(0,\bI_d)$.
\STATE Sample $Z_i, i=1,\ldots,m$, from $ N(0,\bI_d)$.
\STATE Compute $\tilde{b}_{m}$ according to \eqref{drifte1} or \eqref{drifte2},
\STATE $
\widetilde{Y}_{t_{k+1}}=\widetilde{Y}_{t_{k}} + \sigma s \tilde{b}_{m}\left(\widetilde{Y}_{t_{k}}, t_{k}\right) + \sqrt{\sigma s} \epsilon_{k+1}$.
\ENDFOR
\STATE Output:  $\{\widetilde{Y}_{t_k}\}_{k=1}^{K}$.
\end{algorithmic}
\end{algorithm}
In the next section, we establish   the probability bound of $\widetilde{Y}_{t_K}$ being a $\tau$-global minimizer (Theorem \ref{th2}),  and derive the bound in the  Wasserstein-2 distance between  the  law of  $\widetilde{Y}_{t_K}$  generated from Algorithm \ref{alg:1} and the probability distribution measure $\mu_{\sigma}$ under some certain conditions (Theorem \ref{th3}).
\section{Theoretical Property}\label{Theorey}
In this section, we  show that the Gibbs measure $\mu_{\sigma}$ weakly converges to a multidimensional distribution concentrating
on the optimal  points $\{x_1^*,\ldots,x_\kappa^*\}$. Since the minimum value of $V$ is 0, then we estimate the probabilities of $V(X_1)>\tau$ and $V(\widetilde{Y}_{t_k})>\tau$ for any $\tau >0$, and  establish the non-asymptotic bounds  between the law of the samples generated from Algorithm \ref{alg:1} and the target distribution $\mu_{\sigma}$ in the Wasserstein-2 distance.
Recall that the linear growth condition \eqref{cond1}  and Lipschitz continuity \eqref{cond2} hold under conditions (\textbf{P1}) and (\textbf{P2}), which make the Schr{\"o}dinger-F\"{o}llmer SDE \eqref{sch-equation} have the unique  strong solution.
Besides, we  assume that the drift term $b(x,t)$ is Lipschitz continuous in $x$ and $\frac{1}{2}$-H{\"o}lder continuous in $t$,
\begin{align}\label{cond3}
\|b(x,t)-b(y,s)\|_2\leq  C_{2,\sigma} \left(\|x-y\|_2+|t-s|^{\frac{1}{2}}\right), \tag{C3}
\end{align}
where 
 $C_{2,\sigma}$ 
is a positive constant depending on $\sigma$.
\begin{remark}\label{R3}
\eqref{cond1} and \eqref{cond2} are the essentially sufficient
conditions such that  the
Schr{\"o}dinger-F\"{o}llmer SDE \eqref{sch-equation}
admits the unique strong solution.
 \eqref{cond3} has  been introduced in  Theorem 4.1 of \cite{tzen2019theoretical}
 and it is also similar to the condition (H2) of \cite{chau2019stochastic} and Assumption 3.2 of \cite{barkhagen2018stochastic}.
 Obviously,  \eqref{cond3} implies  \eqref{cond2},  and \eqref{cond1} holds if the drift term $b(x,t)$ is bounded over $\mathbb{R}^d\times[0,1]$. 
\end{remark}

Firstly,  we show that the Gibbs measure $\mu_{\sigma}$ weakly converges to a multidimensional distribution. This result can be traced back to the 1980s. For the overall continuity of the article, we combine the Laplace's method in
\cite{hwang1980laplace,hwang1981generalization} to give a detailed proof of the result. The key idea is to prove that for each $\delta'>0$  small enough, $ \mu_{\sigma}(\{x;\|x-x_i^*\|_2 <\delta' \})$ converges to $\dfrac{\left(\det \nabla^2 V(x_i^*) \right)^{-\frac 1 2}}{\sum_{j=1}^\kappa \left(\det \nabla^2 V(x_j^*) \right) ^{-\frac 1 2}}$ as $\sigma \downarrow 0$.

Next, we want to estimate the probabilities of $V(X_1)>\tau$ and $V(\widetilde{Y}_{t_K})>\tau$ for any $\tau >0$.  However, the second analysis is more complicated due to the discretization, and the main idea comes from \cite{dalalyan2017theoretical,cheng2018underdamped}, which constructs a continuous-time interpolation 
SDE
for the Euler-Maruyama discretization. In their works, the relative entropy is controlled via using the Girsanov's theorem to estimate Radon-Nikodym derivatives.

Another method of controlling the relative entropy is proposed by
\cite{mou2019improved}.
By direct calculations, the time derivative of the relative entropy between the interpolated and the original 
SDE
\eqref{sch-equation} is controlled by the mean squared difference between the drift terms of the Fokker-Planck equations for the original and the interpolated processes. Compared to the bound obtained from Lemma \ref{lemma 3.2}, the  bound in \cite{mou2019improved} has an additional backward conditional expectation inside the norm. It becomes a key reason for obtaining higher precision orders. But it must satisfy the dissipative condition to the drift term of the SDE and initial distribution smoothness.

The concrete results are showed in the following Theorems \ref{th1}, \ref{th2}, \ref{th3}. See Appendix \ref{append} for detailed proofs.

\begin{theorem}\label{th1}
Let $V: \mathbb R^d \rightarrow \mathbb R$ be a twice continuously differentiable  function. Suppose there exists a finite set $N:=\{x\in \mathbb R^d;V(x)=\inf_x V(x) \}=\{x_1^*,\ldots, x_\kappa^*\}, \kappa \ge 1$ and $\int_{\mathbb R^d} \exp(-V(x))dx< \infty$, then
\begin{equation*}
\mu_{\sigma} \xrightarrow{w} \dfrac{\sum_{i=1}^\kappa \left(\det \nabla^2 V(x_i^*) \right) ^{-\frac 1 2}\delta_{x_i^*}}{\sum_{j=1}^\kappa \left(\det \nabla^2 V(x_j^*) \right) ^{-\frac 1 2}}, \quad \text{as $\sigma \downarrow 0.$}
\end{equation*}
\end{theorem}

Under  Theorem \ref{th1},  a natural question is to consider the convergence rate about the measure $\mu_{\sigma}$ converging to  the multidimensional distribution. 
To that end, we  apply the tools from the large deviation
of the Gibbs measure which is absolutely continuous with respect to Lebesgue measure, 
see \cite{chiang1987diffusion,holley1989asymptotics,marquez1997convergence}.
Therefore we can obtain the following property.
\begin{proposition}\label{prop1}
Assume that the conditions of Theorem \ref{th1} hold, then for all $\tau>0$,
\begin{equation*}
\lim_{\sigma \to 0} \sigma \log \mu_{\sigma}\left(V(x)-\min_x V(x) \ge \tau \right)=-\tau.
\end{equation*}
\end{proposition}

\begin{remark}
Although we can directly use the large deviation principle to obtain Proposition \ref{prop1}, further, we can conclude that the Gibbs measure $\mu_{\sigma}$  weakly converges to the global minimum points of the potential function $V$ and obtain the corresponding convergence rate. However, we can not directly obtain the specific limit distribution form 
Proposition \ref{prop1}.
\end{remark}
\begin{theorem}\label{th2}
Suppose  $(\textbf{A})$ 
holds. Then, for each $\varepsilon \in (0, \tau), \sigma \in (0,1)$, there exists a constant $C_{\tau,\varepsilon,d}$ (depending on $\tau, \varepsilon, d$)  given in \eqref{Const1} such that
\begin{align}
& \mathbb P(V(X_1)>\tau)\leq C_{\tau,\varepsilon,d} \exp\left(-\frac{\tau-\varepsilon}{\sigma} \right),\notag\\ 
& \mathbb P(V(\widetilde{Y}_{t_K})>\tau) \leq C_{\tau,\varepsilon,d} \exp\left(-\frac{\tau-\varepsilon}{\sigma} \right)+
 C^{\sharp}_{1,\sigma} \sqrt{d(2d+3) s} +C^{\sharp}_{2,\sigma} \sqrt{\frac{4d}{m}}, \label{eq-thm3.5-2}
\end{align}
where 
\begin{align*}
& C^{\sharp}_{1,\sigma} := \frac{\gamma_{\sigma}}{\xi_{\sigma}}+\frac{\gamma^3_{\sigma}}{\xi^3_{\sigma}}, \quad \frac{\gamma_{\sigma}}{\xi_{\sigma}}=\left\lbrace \left(\frac{M_{2,R}}{\sigma} \right)^2 +\frac{M_{3,R}}{\sigma} \right \rbrace \exp\left(\frac{M_{1,R}-m_{1,R}}{\sigma} \right), \\
& C^{\sharp}_{2,\sigma} :=\frac{\gamma^2_{\sigma}}{\xi^2_{\sigma}} +\frac{\gamma_{\sigma} \zeta_{\sigma}}{\xi^2_{\sigma}}, \quad \frac{\gamma_{\sigma}}{\xi_{\sigma}}=\left\lbrace \left(\frac{M_{2,R}}{\sigma} \right)^2 +\frac{M_{3,R}}{\sigma} \right \rbrace \exp\left(\frac{M_{1,R}-m_{1,R}}{\sigma} \right),
\end{align*}
$M_{1,R}$, $M_{2,R}$, $M_{3,R}$ and $m_{1,R}$ are given in \eqref{eqn:M1-M3}-\eqref{eqn:m1}, and $\zeta_{\sigma}=\exp(M_{1,R}/ \sigma)$. 
\end{theorem}

\begin{remark}
 The first term in the right hand side of \eqref{eq-thm3.5-2} originates from the difference between Gibbs sampling and optimization, which is determined by the simulated annealing method itself. The latter two terms are caused by Euler discretization of SDE \eqref{sch-equation} and Monte Carlo  estimation of drift coefficient $b(x,t)$ in \eqref{driftb}. The latter two terms of \eqref{eq-thm3.5-2} 
  depend on $d$ polynomially and $\sigma$ exponentially.
This result is  contrary to the  some Langevin methods, that is, 
the related convergence error bounds depend on $d$ exponentially \citep{wang2009log,hale2010asymptotic,menz2014poincare,raginsky2017non}
implying that the efficiency of Langevin samplers may suffer from the curse of dimensionality.

By Theorem \ref{th2}, 
$\forall \ \  0 < \delta \ll 1$,  with probability at least $1-\sqrt{\delta}$,  $\widetilde{Y}_{t_K} $ is a $\tau$-global minimizer of $V$, i.e.,
$V(\widetilde{Y}_{t_K})\leq \tau + \inf V(x)$, 
if the number of iterations $K\geq \mathcal{O}\left(\frac{d^2}{\delta}\right)$,  the number of Gaussian samples per iteration  $m \geq \mathcal{O}\left(\frac{d}{\delta}\right)$ and $\sigma \leq \mathcal{O} \left(\frac{\tau}{\log(1/\delta)}\right)$. 
\end{remark}

At last, we establish the non-asymptotic bounds 
between the law of the samples generated from Algorithm \ref{alg:1} and the  distribution $\mu_{\sigma}$ in the Wasserstein-2 distance.
We introduce the definition of Wasserstein distance.
Let $\mathcal{D}(\nu_1, \nu_2)$ be the collection of coupling probability measures on $\left(\mathbb{R}^{2d},\mathcal{B}(\mathbb{R}^{2d})\right)$ such that its respective marginal distributions are $\nu_1$ and $\nu_2$. The Wasserstein distance of order $p \geq 1$ measuring the discrepancy between $\nu_1$  and $\nu_2$ is defined as
\begin{align*}
W_{p}(\nu_1, \nu_2)=\inf _{\nu \in \mathcal{D}(\nu_1, \nu_2)}\left(\int_{\mathbb{R}^{d}} \int_{\mathbb{R}^{d}}\left\|\theta_1-\theta_2\right\|_2^{p} \mathrm{d} \nu\left(\theta_1,\theta_2\right)\right)^{1/p}.
\end{align*}

\begin{theorem}\label{th3}
Assume (\textbf{A}) and \eqref{cond3} hold,
then
\begin{align*}
W_2(\text{Law}(\widetilde{Y}_{t_K}),\mu_{\sigma})\leq C^{\sharp}_{\sigma} \left ( C^{\sharp}_{3,\sigma} \sqrt{s} +C^{\sharp}_{2,\sigma} \sqrt{\frac{16d}{m}}  \right ),
\end{align*}
where 
\begin{align*}
& C^{\sharp}_{\sigma} :=
\exp(1/2+8C_{2,\sigma}^2), \quad  C^{\sharp}_{3,\sigma} :=2C_{1,\sigma} +4C^2_{1,\sigma} \left \lbrace 1+ C_{0,\sigma} \exp\left(2\sqrt{C_{0,\sigma}} +1 \right) \right \rbrace^{\frac 1 2}, \\
& C^{\sharp}_{2,\sigma} :=\left(\frac{\gamma_{\sigma}}{\xi_{\sigma}} \right)^2 + \frac{\gamma_{\sigma} \zeta_{\sigma}}{\xi^2_{\sigma}}, \quad  \frac{\gamma_{\sigma}}{\xi_{\sigma}}=\left\lbrace \left(\frac{M_{2,R}}{\sigma} \right)^2 +\frac{M_{3,R}}{\sigma} \right \rbrace \exp\left(\frac{M_{1,R}-m_{1,R}}{\sigma} \right),
\end{align*}
with $C_{0,\sigma} =C_{1,\sigma}=\gamma_{\sigma}/\xi_{\sigma} + \gamma^2_{\sigma}/\xi^2_{\sigma}$.
\end{theorem}
\begin{remark}\label{rem3.10}
 Langevin sampling method has been studied under  the (strongly) convex potential assumption \citep{durmus2016sampling, durmus2017nonasymptotic, durmus2016high-dimensional, dalalyan2017further, dalalyan2017theoretical,cheng2018convergence,dalalyan2019user-friendly}.
Also, there are some  mean-field Langevin methods, see 
\cite{garbuno2020interacting,wang2022accelerated}
for more details.
However, the Langevin diffusion process tends to the target distribution as time $t$ goes to infinity while the Schr\"{o}dinger bridge achieve this  in unit time.
The Schr\"{o}dinger bridge has been shown to have close connections with a number of problems in statistical physics,  optimal transport  and optimal control \citep{leonard2014survey}. However,  only a few recent works  use this tool for statistical sampling.
\cite{bernton2019schr} proposed  the Schr\"{o}dinger bridge sampler
 by applying  the iterative proportional fitting  method  (Sinkhorn algorithm \citep{sinkhorn1964relationship,peyre2019computational}).
 For the Gibbs measure  $\mu_{\sigma}$,
Schr{\"o}dinger bridge samplers iteratively modifies the transition kernels of the reference Markov chain to obtain a process whose marginal distribution at terminal time  is approximate to $\mu_{\sigma}$, via regression-based approximations of the corresponding iterative proportional fitting recursion.

\end{remark}

\section{Simulation studies}\label{simulation}
In this section, we conduct numerical simulations to evaluate the performance of the proposed method (Algorithm \ref{alg:1}), and compare it  with the Langevin method.
We  consider the following non-convex smooth function \citep{carrillo2021consensus}, which maps from $\mathbb{R}^d$ to $\mathbb{R}$,
$$
V(x)=\frac{1}{d} \sum_{i=1}^{d}\left[\left(x_{i}-B\right)^{2}-10 \cos \left(2 \pi\left(x_{i}-B\right)\right)+10\right]+C,
$$
with $B=\operatorname{argmin}_{x\in \mathbb{R}^d} V(x)$ and $C=\min_{x\in \mathbb{R}^d} V(x)$.
Figure \ref{f1} depicts this target function $V$ with setting $B=0,C=0,d=1$, and we can see that the number of local minimizers  in $\mathbb{R}$ is infinite and only one global minimizer exists, i.e, $x=0$.
In the numerical experiment, 
we consider the case $d=2,B=0,C=0$ in $V$.
We set  $K=200,m=1000,\sigma=0.01$ in Algorithm \ref{alg:1}, and Langevin method is implemented by R package yuima \citep{iacus2018simulation}.
As shown in Proposition \ref{SBP},
the target distribution can be  exactly
obtained  at time one.
Thus, we only need to keep the last  data in the  Euler-Maruyama discretization of Schr{\"o}dinger-F{\"o}llmer diffusion in each iteration and repeat this scheme such that the desired sample size is derived, i.e., we get one sample when running Algorithm \ref{alg:1} one time (it costs $200\times 1000$ Gaussian samples at each run).
In comparison,  the Langevin diffusion \eqref{sde} tends to the target distribution as time $t$ goes to infinity, then it should burn out sufficient data points empirically  in each  Euler-Maruyama discretization. To make the comparison fair, we run the Langevin method  50 times. At each run we generate $200\times 1000$ samples and keep  the one  with the best function value to show. 
Figure \ref{f2} shows the simulation results of 50 independent runs, where the red line  yields the global minimizer (GM) $(x_1=0,x_2=0)$.
From Figure \ref{f2},
we can conclude that our  proposed method   obtains  the global minimizer approximately and performs better than the Langevin method.
\begin{figure}[ht!]
\centering
\includegraphics[height=5.5cm,width=8.5cm]{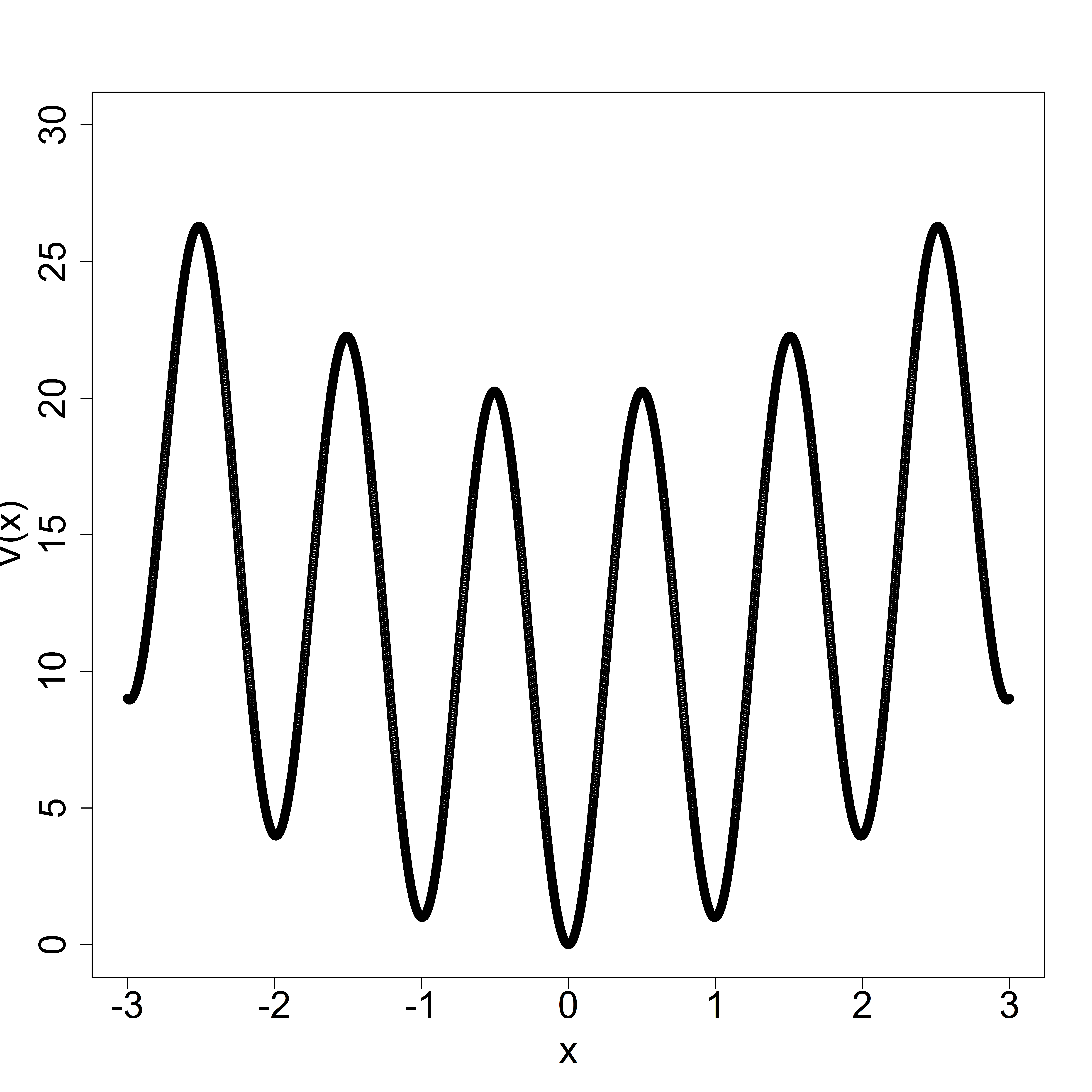}
\caption{$B=0$, $C=0$, $d=1$.}
\label{f1}
\end{figure}
\begin{figure}[ht!]
\centering
\begin{minipage}[c]{0.55\textwidth}
\centering
\includegraphics[height=5.5cm,width=8.5cm]{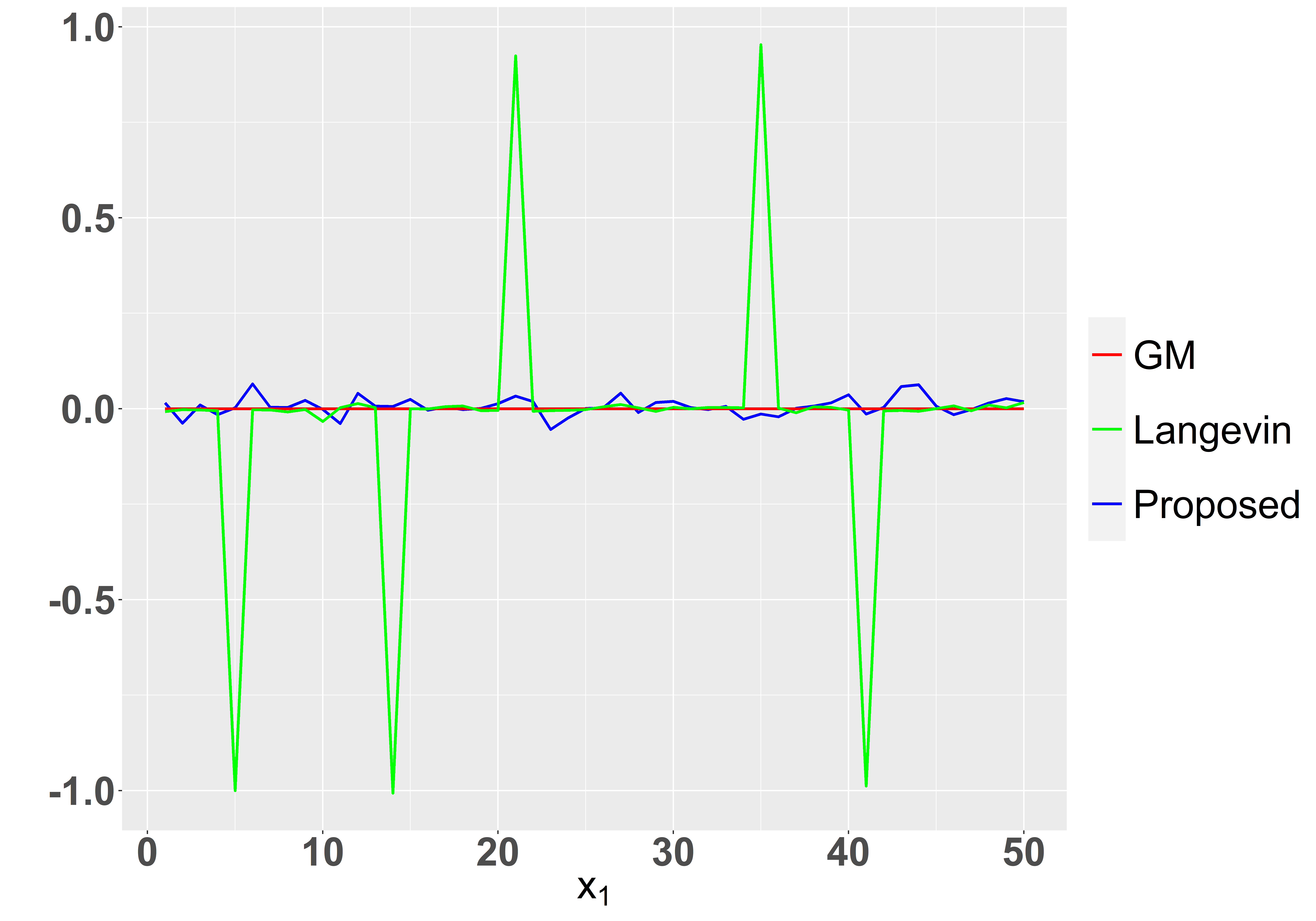}
\end{minipage}%
\begin{minipage}[c]{0.55\textwidth}
\centering
\includegraphics[height=5.5cm,width=8.5cm]{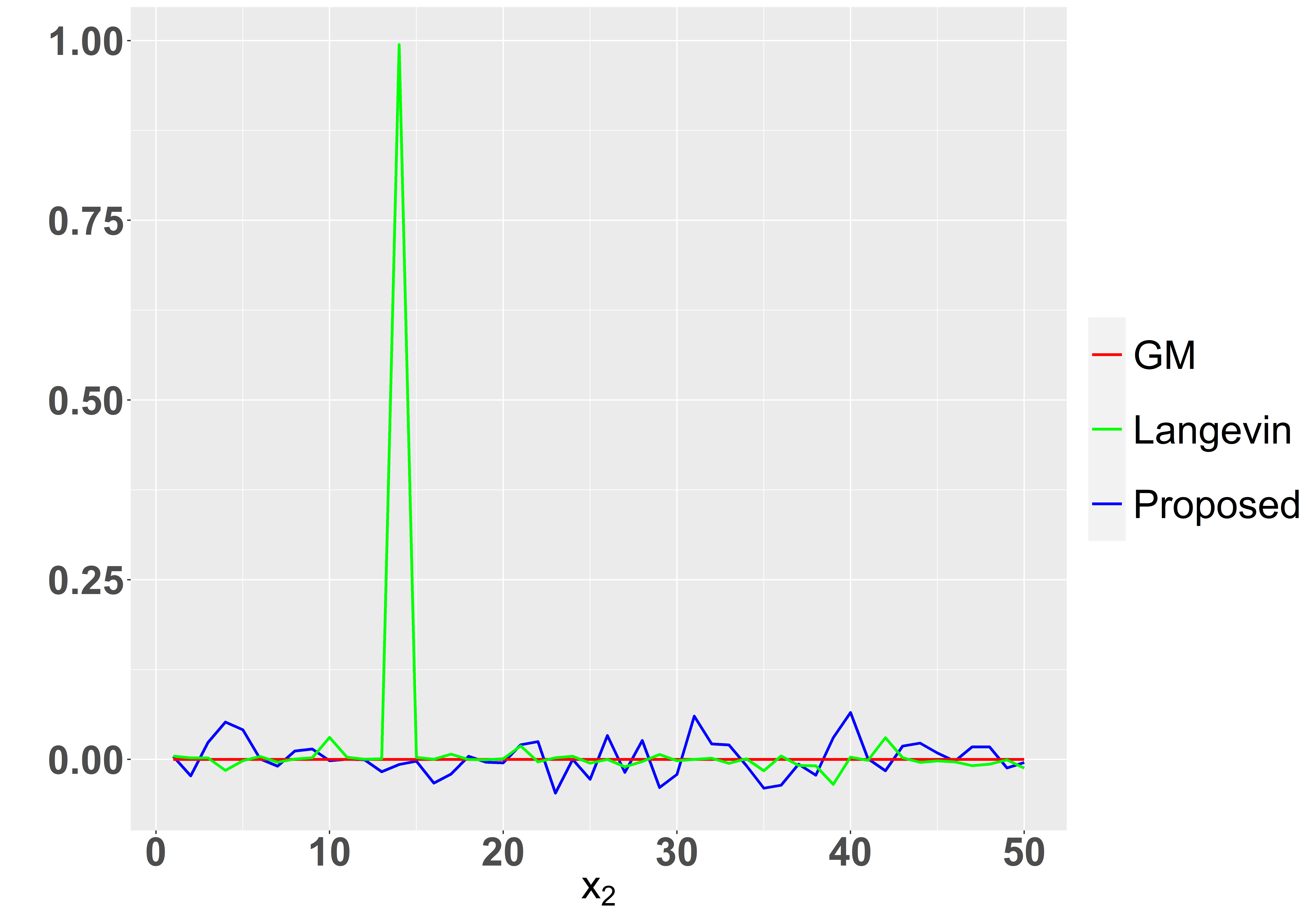}
\end{minipage}
\caption{Proposed method VS Langevin method.}
\label{f2}
\end{figure}

\section{Conclusion}\label{conlusion}
We study the problem of finding global minimizers of $V(x):\mathbb{R}^d\rightarrow\mathbb{R}$ approximately  via sampling from a probability distribution $\mu_{\sigma}$ with density $p_{\sigma}(x)=\dfrac{\exp(-V(x)/\sigma)}{\int_{\mathbb R^d} \exp(-V(y)/\sigma) dy }$ with respect to the Lebesgue measure for  $\sigma \in (0,1]$ small enough.
We analyze a sampler based on the Euler discretization scheme of the Schr{\"o}dinger-F{\"o}llmer diffusion processes with stochastic approximation under appropriate assumptions on the step size $s$ and the potential $V$. We prove that the output of the proposed sampler is an approximate global minimizer of $V(x)$  with high probability.
Moreover, simulation studies
verify the effectiveness of  the proposed method  on solving non-convex smooth optimization problems and it
performs better than the Langevin method.

\section{Acknowledgments}\label{section acknowledgments}
We would like the thank the anonymous referees for their useful comments and suggestions, which have led to considerable
improvements in the paper. This work is supported by the National Key Research and Development Program
of China (No. 2020YFA0714200), by the National Science Foundation of China (No.
12125103, No. 12071362, No. 11871474, No.11871385). The numerical calculations
have been done at the Supercomputing Center of Wuhan University.

\section{appendix}\label{append}
\setcounter{equation}{0}
\def\theequation{A.\arabic{equation}}
In the appendix, we will show   (\textbf{P1})-(\textbf{P2}) and  prove the all  Propositions and Theorems, i.e.,
Propositions  \ref{SBP}, \ref{prop1},
Theorems \ref{th1},  \ref{th2}, \ref{th3}.

\subsection{Verify  (\textbf{P1})-(\textbf{P2})}
\label{vc12}
\begin{proof}
Recall the definition of  $f_{\sigma}$ in \eqref{drb} and the assumption   $V(x)=\|x\|_2^2/2$ when $\|x\|_2>R$. Through some simple calculations, we have
\begin{align*}
\nabla \hat f_{\sigma}(x)&=\exp\left(\frac{\|x\|_2^2}{2\sigma}-\frac{V(x)}{\sigma}\right)\cdot \left(\frac{x}{\sigma}-\frac{\nabla V(x)}{\sigma}\right),\\
\nabla^2 \hat f_{\sigma}(x)&=\exp\left(\frac{\|x\|_2^2}{2\sigma}-\frac{V(x)}{\sigma}\right)\cdot \left(\frac{x}{\sigma}-\frac{\nabla V(x)}{\sigma}\right)\cdot\left(\frac{x}{\sigma}-\frac{\nabla V(x)}{\sigma}\right)^{\top}\\
&~~~~+ \exp\left(\frac{\|x\|_2^2}{2\sigma}-\frac{V(x)}{\sigma}\right) \cdot\left(\frac{\bI_d}{\sigma} -\frac{\nabla^2 V(x)}{\sigma}\right).
\end{align*}
Then,  the properties  (\textbf{P1})-(\textbf{P2})  hold if for each $\sigma \in(0,1]$,
\begin{align*}
& \lim _{\widetilde R \rightarrow \infty} \sup _{\|x\|_2 \geq \widetilde R} \exp \left(-\frac{V(x)}{\sigma}+\frac{\|x\|_2^{2}}{2 \sigma}\right)\left\|\frac{ \mathbf{I}_d-\operatorname{\nabla^2}V(x)}{\sigma}\right\|_2 \\ 
\le & \lim _{\widetilde R \rightarrow \infty} \sup _{\|x\|_2 \geq \widetilde R} \exp \left(-V(x) + \frac{\|x\|_2^{2}}{2}\right)\left\| \mathbf{I}_d-\operatorname{\nabla^2}V(x)\right\|_2 <\infty,
\end{align*}
and
\begin{align*}
& \lim _{\widetilde R \rightarrow \infty} \sup _{\|x\|_2 \geq \widetilde R} \exp \left(-\frac{V(x)}{\sigma}+\frac{\|x\|_2^{2}}{2\sigma}\right) \cdot\left\|\frac{x- \nabla V(x)}{\sigma} \right\|_2 \\
\le & \lim _{\widetilde R \rightarrow \infty} \sup _{\|x\|_2 \geq \widetilde R} \exp \left(-V(x)+\frac{\|x\|_2^{2}}{2}\right) \cdot\left\|x-\nabla V(x) \right\|_2
<\infty,
\end{align*}
where we use the fact that $a \mapsto a\exp(-ax)$ is non-increasing for $a \ge 1, x\in \mathbb R_{+}$. 
\end{proof}

\subsection{Proof of Proposition \ref{SBP}}
\begin{proof}
By (\textbf{P1}) and (\textbf{P2}),
it yields that for all $x \in \mathbb{R}^d$ and $t \in[0,1]$,
\begin{align}\label{re3}
\|b(x, t)\|_2=\frac{\left\|\nabla Q_{1-t} \hat f_{\sigma}(x)\right\|_2}{Q_{1-t} \hat f_{\sigma}(x)} \leq \frac{\gamma_{\sigma}}{\xi_{\sigma}}.
\end{align}
Then, by (\textbf{P1})-(\textbf{P2}) and \eqref{re3},  for all  $x, y \in \mathbb{R}^d$ and $t \in[0,1]$,
\begin{align*}
\left\|b(x, t)-b\left(y, t\right)\right\|_2 &=\left\|\frac{\nabla Q_{1-t} \hat f_{\sigma}(x)}{Q_{1-t} \hat f_{\sigma}(x)}-\frac{\nabla Q_{1-t} \hat f_{\sigma} \left(y\right)}{Q_{1-t} \hat f_{\sigma}\left(y\right)}\right\|_2 \\
&=\left\| \frac{\nabla Q_{1-t}\hat f_{\sigma}(x)-\nabla Q_{1-t}\hat f_{\sigma}(y)}{Q_{1-t}\hat f_{\sigma}(y)} +\frac{\nabla Q_{1-t} \hat f_{\sigma}(x)\left(Q_{1-t} \hat f_{\sigma}(y)-Q_{1-t} \hat f_{\sigma}(x)\right)}{Q_{1-t} \hat f_{\sigma}(x)Q_{1-t} \hat f_{\sigma}(y)} \right\|_2  \\
& \leq \frac{\left\|\nabla Q_{1-t} \hat f_{\sigma}(x)-\nabla Q_{1-t} \hat f_{\sigma} \left(y\right)\right\|_2}{Q_{1-t} \hat f_{\sigma} \left(y\right)}+\|b(x, t)\|_2 \cdot \frac{\left|Q_{1-t} \hat f_{\sigma} (x)-Q_{1-t} \hat f_{\sigma}\left(y\right)\right|}{Q_{1-t} \hat f_{\sigma}\left(y\right)} \\
&\leq\left(\frac{\gamma_{\sigma}}{\xi_{\sigma}}+\frac{\gamma_{\sigma}^{2}}{\xi_{\sigma}^2}\right)\left\|x-y\right\|_2.
\end{align*}
Setting $C_{1,\sigma} := \frac{\gamma_{\sigma}}{\xi_{\sigma}}+\frac{\gamma_{\sigma}^{2}}{\xi_{\sigma}^2}$
yields the Lipschitiz continuous condition \eqref{cond2}. Combining \eqref{re3} and \eqref{cond2} with the triangle inequality, we have
$$
\|b(x,t)\|_2\leq \|b(0,t)\|_2+C_1\|x\|_2\leq \frac{\gamma_{\sigma}}{\xi_{\sigma}}+C_{1,\sigma} \|x\|_2.
$$
Let $C_{0,\sigma} := \max\left\{\frac{\gamma_{\sigma}}{\xi_{\sigma}}, C_{1,\sigma} \right\}$, then
\eqref{cond1} holds. Therefore, the drift term $b(x,t)$ satisfies the linear grow condition \eqref{cond1}  and Lipschitz condition \eqref{cond2}, then the  Schr\"{o}dinger-F\"{o}llmer diffusion SDE \eqref{sch-equation} has the unique strong solution \citep{revuz2013continuous,pavliotis2014stochastic}.

Moreover, Schr\"{o}dinger-F\"{o}llmer diffusion process $\{X_t\}_{t\in [0,1]}$ defined in \eqref{sch-equation} admits the transition probability density
\begin{align*}
p_{s, t}(x, y) :=\widetilde{p}_{s, t}(x, y) \frac{Q_{1-t} f_{\sigma}(y)}{Q_{1-s} f_{\sigma}(x)},
\end{align*}
where
\begin{align*}
\widetilde{p}_{s, t}(x, y)=\frac{1}{(2 \pi \sigma (t-s))^{d/2}} \exp \left(-\frac{1}{2\sigma(t-s)}\|x-y\|^{2}_2\right)
\end{align*}
is the transition probability density of a $d$-dimensional Brownian motion $\sqrt{\sigma} B_{t}$. See \cite{dai1991stochastic,lehec2013representation} for details.
It follows that for any Borel measurable set $A \in \mathcal{B}(\mathbb{R}^d)$,
\begin{align*}
\mathbb P(X_1\in A)&=\int_A p_{0,1}(0,y) \mathrm{d} y\\
&=\int_A \widetilde{p}_{0,1}(0,y)\frac{Q_{0} f_{\sigma}(y)}{Q_{1} f_{\sigma}(0)}\mathrm{d} y \\
&=\mu_{\sigma}(A),
\end{align*}
where the last equality follows from $Q_1 f_{\sigma}(0)=\mu_{\sigma}(\mathbb R^d)=1$ and $Q_0 f_{\sigma}(y)=f_{\sigma}(y)$. Therefore, $X_1$ is distributed as the probability distribution $\mu_{\sigma}$.
This completes the proof.
\end{proof}

\subsection{Proof of Proposition \ref{prop1}}
\begin{proof}
Under the Theorem \ref{th1}, then $C_{\sigma}=\int_{\mathbb R^d} \exp\left(-V(x)/\sigma \right)dx \le \int_{\mathbb R^d} \exp\left(-V(x)\right)dx<\infty$ for all $\sigma \in (0,1]$.
According to the Varadhan's theorem 1.2.3 in \cite{dupuis2011weak}, it follows that the family $\{\mu_{\sigma} \}_{\sigma \in [0,1]}$
 on $\mathcal B(\mathbb R^d)$ satisfies large deviation principle with rate function $V(x)-\min_x V(x)$, that is,
 for every function $F \in C_b(\mathbb R^d)$, the bounded continuous function space on $\mathbb R^d$,
\begin{equation}\label{eq A2}
\lim_{\sigma \to 0} \sigma \log \int_{\mathbb R^d}C_{\sigma}^{-1} \exp\left(\frac{F(x)-V(x)}{\sigma}\right) dx=\sup_x \{F(x)-I(x)\},
\end{equation}
where the rate function $I(x)$ is defined by
\begin{equation*}
I(x) := V(x)-\min_x V(x).
\end{equation*}
Next, we only need to prove \eqref{eq A2}. On the one hand,
\begin{equation*}
 \log \int_{\mathbb R^d}C_{\sigma}^{-1} \exp\left(\frac{F(x)-V(x)}{\sigma}\right) dx=-\log \int_{\mathbb R^d} \exp\left( -\frac{V(x)}{\sigma} \right) dx + \log \int_{\mathbb R^d}\exp\left(\frac{F(x)-V(x)}{\sigma}\right) dx.
 \end{equation*}
By  Lemma \ref{lemma 2.1}, we have
\begin{align}\label{eq A3}
-\sigma \log \int_{\mathbb R^d} \exp\left( -\frac{V(x)}{\sigma} \right) dx &=-\sigma \log \int_{\mathbb R^d} \exp\left( -\frac{V(x)-\min_x V(x)}{\sigma} \right) dx + \min_x V(x) \notag \\
& \rightarrow \min_{x} V(x) + \frac d 2 \lim_{\sigma \to 0} \sigma \log \sigma=\min_{x} V(x), \quad \text{as $\sigma \downarrow 0$.}
\end{align}
On the other hand, 
for any positive $K>0$, 
combining Lemma 2.2 in  \cite{varadhan1966asymptotic} and the dominated convergence theorem yields that
\begin{align}
\lim_{\sigma \to 0} \sigma \log \int_{\mathbb R^d}\exp\left(\frac{F(x)-V(x)}{\sigma}\right) dx & = \lim_{K\to \infty} \lim_{\sigma \to 0} \sigma \log \int_{\mathbb R^d}\exp\left(\frac{F(x)-V(x) \land K}{\sigma}\right) dx \notag \\
&= \lim_{K\to \infty} \sup_x \{F(x)-V(x)\land K \}=\sup_x \{F(x)-V(x) \}. \label{eq A4}
\end{align}
Hence, by \eqref{eq A3} and \eqref{eq A4}, we get
\begin{equation*}
\lim_{\sigma \to 0} \sigma \log \int_{\mathbb R^d}C_{\sigma}^{-1} \exp\left(\frac{F(x)-V(x)}{\sigma}\right) dx=\sup_x \{F(x)-V(x)+ \min_x V(x)\}.
\end{equation*}
Since the measure $\mu_{\sigma}$ satisfies the large deviation principle with rate function $I$, if we take the closed set $F:=\{x\in \mathbb R^d;V(x)-\min_x V(x) \ge \tau \}$, then
\begin{equation*}
\lim_{\sigma \to 0} \sigma \log \mu_{\sigma}(F)=\lim_{\sigma \to 0} \sigma \log \mu_{\sigma}\left(V(x)-\min_x V(x) \ge \tau \right)=-\inf_{x \in F} I(x)=-\tau.
\end{equation*}
\end{proof}
\subsection{Preliminary lemmas for Theorem \ref{th1} and Theorem \ref{th2}}
In order to prove that the Gibbs measure $\mu_{\sigma}$ weakly converges to a multidimensional distribution and estimate the probabilities of $V(X_1)>\tau$ and $V(\widetilde{Y}_{t_k})>\tau$ for any $\tau >0$, we first need to prove the following Lemmas \ref{lemma 2.1}-\ref{lemma 3.2}.
\begin{lemma}\label{lemma 2.1}
Assume $V$ is  twice continuously differentiable and satisfies $V(x_0)=0, \nabla V(x_0)=0$, and the Hessian matrix $\nabla^2 V(x_0)$ is positive definite. When $\delta>0$ is sufficiently small, then 
\begin{equation*}
\lim_{t \to +\infty} t^{\frac d 2} \int_{U_{\delta}(x_0)} e^{-t V(x_1, \ldots, x_d)} dx_1 \cdots dx_d =\frac{(2\pi)^{\frac d 2}}{\left(\det \nabla^2 V(x_0) \right)^{\frac 1 2}}
\end{equation*}
for any $x\in U_{\delta}(x_0):=\{x\in \mathbb R^d; \|x-x_0\|_2<\delta\}$.
\end{lemma}
\begin{proof}
For each $\varepsilon \in (0,1)$, there exists 
$\delta>0$ such that
\begin{equation*}
\frac{(1-\varepsilon)(x-x_0)^{\top}\nabla^2 V(x_0)(x-x_0)}{2} \le V(x)\le \frac{(1+\varepsilon)(x-x_0)^{\top}\nabla^2 V(x_0)(x-x_0)}{2}
\end{equation*}
holds for any $x\in U_{\delta}$. Moreover, we have
\begin{equation}\label{eq 11}
\int_{U_{\delta}(x_0)} \exp\left(-tV(x)\right)dx \le \int_{U_{\delta}(x_0)} \exp\left(-\frac t 2 (1-\varepsilon)(x-x_0)^{\top}\nabla^2V(x_0)(x-x_0) \right) dx.
\end{equation}
There is an orthogonal matrix $P$ such taht $P^{\top} \nabla^2 V(x_0)P=\mathrm{diag}(\lambda_1,\ldots,\lambda_d)$, where $\lambda_1, \ldots, \lambda_d$ are the eigenvalues of Hessian matrix $\nabla^2 V(x_0)$. And we denote $(y_1,\ldots,y_d) :=P(x-x_0)$, then  $$(x-x_0)^{\top}\nabla^2V(x_0)(x-x_0)=\sum_{i=1}^{d}\lambda_i(y_i)^2.$$ 
Hence
\begin{align}\label{eq 12}
 \int_{U_{\delta}(x_0)} \exp\left(-\frac t 2 (1-\varepsilon)(x-x_0)^{\top}\nabla^2V(x_0)(x-x_0) \right) dx
 =\int_{U_{\delta}(0)} \exp\left(-\frac{t(1-\varepsilon)}{2}\sum_{i=1}^{d}\lambda_i(y_i)^2 \right) dy.
\end{align}
Further, we can get
\begin{equation}\label{eq 13}
\int_{U_{\delta}(0)} \exp\left(-\frac{t(1-\varepsilon)}{2}\sum_{i=1}^{d}\lambda_i(y_i)^2 \right) dy= \left(\frac{2}{t(1-\varepsilon)} \right)^{\frac d 2} \int_{\|z \|_2\le \sqrt{\frac{t(1-\varepsilon)}{2}}\delta} \exp\left(-\sum_{i=1}^d \lambda_i(z_i)^2 \right)dz,
\end{equation}
where the  equality  holds by 
setting $z=\sqrt{\dfrac{t(1-\varepsilon)}{2}}y$. 
By \eqref{eq 11}, \eqref{eq 12} and \eqref{eq 13},  it follows that
\begin{align*}
\limsup_{t\to +\infty} t^{\frac d 2} \int_{U_{\delta}(x_0)} e^{-t V(x)} dx &\le \left(\frac{2}{1-\varepsilon} \right)^{\frac d 2}\limsup_{t\to +\infty} \int_{\|z \|_2 <\sqrt{\frac{t(1-\varepsilon)}{2}}\delta} \exp\left(-\sum_{i=1}^d \lambda_i(z_i)^2 \right)dz \\
& \le \left(\frac{2}{1-\varepsilon} \right)^{\frac d 2} \int_{\mathbb R^d} \exp\left(-\sum_{i=1}^d \lambda_i(z_i)^2 \right)dz_1 \cdots dz_d \\
&= \left(\frac{2\pi}{1-\varepsilon} \right)^{\frac d 2} \left(\prod_{i=1}^d \frac{1}{\sqrt{\lambda_i}}\right)=\left(\frac{2\pi}{1-\varepsilon} \right)^{\frac d 2} \frac{1}{\left(\det \nabla^2 V(x_0) \right)^{\frac 1 2}}.
\end{align*}
Similarly, we have
\begin{align*}
\liminf_{t\to +\infty} t^{\frac d 2} \int_{U_{\delta}(x_0)} e^{-t V(x)} dx &\ge \left(\frac{2}{1+\varepsilon} \right)^{\frac d 2}\liminf_{t\to +\infty} \int_{\|z \|_2 <\sqrt{\frac{t(1+\varepsilon)}{2}}\delta} \exp\left(-\sum_{i=1}^d \lambda_i(z_i)^2 \right)dz \\
& \ge \left(\frac{2}{1+\varepsilon} \right)^{\frac d 2} \int_{\mathbb R^d} \exp\left(-\sum_{i=1}^d \lambda_i(z_i)^2 \right)dz_1 \cdots dz_d \\
&= \left(\frac{2\pi}{1+\varepsilon} \right)^{\frac d 2} \left(\prod_{i=1}^d \frac{1}{\sqrt{\lambda_i}}\right)=\left(\frac{2\pi}{1+\varepsilon} \right)^{\frac d 2} \frac{1}{\left(\det \nabla^2 V(x_0) \right)^{\frac 1 2}}.
\end{align*}
Therefore, letting $\varepsilon \rightarrow 0^{+}$, we get
\begin{equation*}
\lim_{t \to +\infty} t^{\frac d 2} \int_{U_{\delta}(x_0)} e^{-t V(x)} dx=(2\pi)^{\frac d 2} \left( \prod_{i=1}^d \frac{1}{\sqrt{\lambda_i}}\right)=\frac{(2\pi)^{\frac d 2}}{\left(\det \nabla^2 V(x_0) \right)^{\frac 1 2}}.
\end{equation*}
\end{proof}
\begin{lemma}\label{lemma 3.2}
Let $X=(X_t,\mathcal F_t), Y=(Y_t,\mathcal F_t)$ be strong solutions of the following two stochastic differential equations
\begin{align*}
& dX_t=a(X_t,t)dt+ \sigma dB_t, \quad t\in [0,1] \\
& dY_t=b(Y_t,t)dt+ \sigma dB_t, \quad Y_0=X_0,~ t\in [0,1],
\end{align*}
and $X_0$ is a $\mathcal F_0$-measureable random variable. In addition, if drift terms $a(X_t,t)$ and $b(X_t,t)$ satisfy $\mathbb E \left[ \exp\left( \int_0^1 \|a(X_t,t)\|_2^2 + \|b(X_t,t) \|_2^2 dt\right) \right]< \infty$, then we have
\begin{equation}\label{equ 1}
\frac{d\mathbb P_Y}{d\mathbb P_X}(X)=\exp\left(\sigma^{-1} \int_0^1 \left\langle b(X_t,t)-a(X_t,t), dB_t\right\rangle -\frac{1}{2\sigma^2} \int_0^1 \|b(X_t,t)-a(X_t,t) \|_2^2 dt \right),
\end{equation}
and the relative entropy of $\mathbb P_X$ with respect to $\mathbb P_Y$ satisfies
\begin{equation*}
\mathbb{D}_{\mathrm{KL}}(\mathbb P_X || \mathbb P_Y)
 = \frac{1}{2\sigma^2} \int_0^1 \mathbb E \left[  \|b(X_t,t)-a(X_t,t) \|_2^2\right] dt,
\end{equation*}
where probability distributions $\mathbb P_X, \mathbb P_Y$ are induced by process $(X_t, 0\le t \le 1)$ and $(Y_t, 0\le t\le 1)$, respectively.
\end{lemma}

\begin{proof}
By the Novikov condition \citep{revuz2013continuous}, 
we know that
$$M_t :=\exp\left(\sigma^{-1} \int_0^t \left\langle b(X_u,u)-a(X_u,u), dB_u\right\rangle -\frac{1}{2\sigma^2} \int_0^t \|b(X_u,u)-a(X_u,u) \|_2^2 du \right)$$
 is exponential martingale and $\mathbb E M_t=1$ for all $t \in [0,1]$. We can denote a new probability measure $\mathbb{Q}$ such that $d\mathbb{Q}=M_1 d\mathbb P$. 
 By Girsanov's theorem \citep{revuz2013continuous}, 
 under the new probability measure $\mathbb Q$, we can conclude that
\begin{equation*}
\widetilde B_t:=B_t- \sigma^{-1} \int_0^t (b(X_u,u)-a(X_u,u))du
\end{equation*}
is a $\mathbb Q$-Brownian motion. Hence, under the new probability measure $\mathbb Q$,
\begin{align*}
b(X_t,t)dt +\sigma d\widetilde B_t &=b(X_t,t)dt +\sigma dB_t -(b(X_t,t)-a(X_t,t))dt \\
                            &=a(X_t,t)dt+ \sigma dB_t=dX_t.
\end{align*}
Thus, we have $\mathbb Q_X=\mathbb P_Y$, where $\mathbb Q_X$ is the distribution of $X$ under the measure $\mathbb Q$. Furthermore, we can obtain \eqref{equ 1}. On the other hand, by the definition of realtive entropy of $\mathbb P_X$ with respect to $\mathbb P_Y$, we have
\begin{equation*}
\mathbb{D}_{\mathrm{KL}}(\mathbb P_X || \mathbb P_Y)
=\mathbb E\left[-\log \left( \frac{d\mathbb P_Y}{d\mathbb P_X}(X)\right) \right]=\frac{1}{2\sigma^2} \int_0^1 \mathbb E \left[\|b(X_t,t)-a(X_t,t) \|_2^2 \right]  dt.
\end{equation*}
Therefore, the proof of Lemma \ref{lemma 3.2} is completed.
\end{proof}

\subsection{Proof of Theorem \ref{th1}}
\begin{proof}
The result can be traced back to the 1980s. For the overall continuity of the article, we use the Laplace's method in
\cite{hwang1980laplace,hwang1981generalization}
to give a detailed proof of the result. The key idea is to prove that for each $\delta'>0$  small enough, $\mu_{\sigma}(\{x;\|x-x_i^*\|_2 <\delta' \})$ converges to $\dfrac{\left(\det \nabla^2 V(x_i^*) \right)^{-\frac 1 2}}{\sum_{j=1}^\kappa \left(\det \nabla^2 V(x_j^*) \right) ^{-\frac 1 2}}$ as $\sigma \downarrow 0$. We firstly introduce the following notations
\begin{align*}
& a(\delta'):=\inf\{V(x);\|x-x_i^* \|_2 \ge \delta', 1\le i\le \kappa \},\\
& \widetilde{m}_i(\sigma, \delta'):= \int_{\|x-x_i^* \|_2 <\delta'} \exp\left(-\frac{V(x)}{\sigma} \right)dx, \quad 1\le i \le \kappa, \\
& \widetilde{m}(\sigma,\delta'):=\int_{\bigcup_{i=1}^\kappa \|x-x_i^* \|\ge \delta'} \exp\left(-\frac{V(x)}{\sigma} \right)dx.
\end{align*}
 Hence, we have
 \begin{equation}\label{eq 14}
 \mu_{\sigma}(\{x;\|x-x_i^*\|_2 <\delta' \})=\dfrac{\int_{\|x-x_i^* \|_2< \delta'}\exp\left(-\frac{V(x)}{\sigma} \right)dx}{\int_{\mathbb R^d}\exp\left(-\frac{V(x)}{\sigma} \right)dx}=\frac{\widetilde{m}_i(\sigma,\delta')}{\sum_{j=1}^\kappa \widetilde{m}_j(\sigma,\delta') +\widetilde{m}(\sigma,\delta')}.
 \end{equation}
 On the one hand, $\nabla^2 V(x_i^*)$ is a positive definite symmetric matrix. For any $\varepsilon \in (0,1)$, we can choose $0 <\delta' <\varepsilon$ such that
\begin{equation*}
 \frac{(x-x_i^*)^{\top}(\nabla^2 V(x_i^*)-\varepsilon \bI_d)(x-x_i^*)}{2} \le V(x)-V(x_i^*) \le  \frac{(x-x_i^*)^{\top}(\nabla^2 V(x_i^*)+\varepsilon \bI_d)(x-x_i^*)}{2}
\end{equation*}
holds for any $\|x-x_i^* \|_2<\delta'$. Thus, for any $i=1,\ldots, \kappa$, we obtain
 \begin{align*}
 &(2\pi \sigma)^{-\frac d 2} e^{\frac{V(x_i^*)}{\sigma}}\widetilde{m}_i(\sigma,\delta') \le (2\pi \sigma)^{-\frac d 2}\int_{\|x-x_i^*\|_2 <\delta'} \exp\left(-\frac{(x-x_i^*)^{\top}(\nabla^2 V(x_i^*)-\varepsilon \bI_d)(x-x_i^*)}{2\sigma} \right)dx, \\
 & (2\pi \sigma)^{-\frac d 2} e^{\frac{V(x_i^*)}{\sigma}}\widetilde{m}_i(\sigma,\delta') \ge (2\pi \sigma)^{-\frac d 2}\int_{\|x-x_i^*\|_2 <\delta'} \exp\left(-\frac{(x-x_i^*)^{\top}(\nabla^2 V(x_i)+\varepsilon \bI_d)(x-x_i^*)}{2\sigma} \right)dx.
 \end{align*}
 By Lemma \ref{lemma 2.1} and letting $\sigma \rightarrow 0$, we have
 \begin{align*}
 \left(\det(\nabla^2 V(x_i^*)+\varepsilon \bI_d)\right)^{-\frac 1 2} &\le \liminf_{\sigma \to 0} (2\pi \sigma)^{-\frac d 2} e^{\frac{V(x_i^*)}{\sigma}}\widetilde{m}_i(\sigma,\delta') \notag \\
 & \le \limsup_{\sigma \to 0} (2\pi \sigma)^{-\frac d 2} e^{\frac{V(x_i^*)}{\sigma}}\widetilde{m}_i(\sigma,\delta') \notag \\
 & \le \left(\det(\nabla^2 V(x_i^*)-\varepsilon \bI_d)\right)^{-\frac 1 2}.
 \end{align*}
 As $\varepsilon \downarrow 0$, we get
 \begin{equation}\label{eq 15}
 \lim_{\sigma \to 0}(2\pi \sigma)^{-\frac d 2} \exp\left( \frac{V(x_i^*)}{\sigma}\right) \widetilde{m}_i(\sigma,\delta')=\left(\det \nabla^2 V(x_i^*) \right)^{-\frac 1 2}.
 \end{equation}
 On the other hand, we have
 \begin{align*}
(2\pi \sigma)^{-\frac d 2} \exp\left( \frac{V(x_i^*)}{\sigma}\right) \widetilde{m}(\sigma,\delta') & =(2\pi \sigma)^{-\frac d 2}\exp \left(-\frac{a(\delta')-V(x_i^*)}{\sigma}\right) \notag \\
& \qquad \quad \times \int_{\bigcup_{i=1}^\kappa \|x-x_i^* \|\ge \delta'} \exp\left(-\frac{V(x)-a(\delta')}{\sigma} \right)dx.
 \end{align*}
 Since $a(\delta'):=\inf\{V(x);\|x-x_i^* \|_2 \ge \delta', 1\le i\le \kappa \}>V(x_i^*)$ and $V(x) \ge a(\delta')$ holds in $\{x\in \mathbb{R}^d: \|x-x_i^*\|_2 \ge \delta'\}$, then for any $\sigma \in (0,1]$, we have
 \begin{align*}
 \int_{\bigcup_{i=1}^\kappa \|x-x_i^* \| \ge \delta'} \exp\left(-\frac{V(x)-a(\delta')}{\sigma} \right)dx &\le \int_{\bigcup_{i=1}^\kappa \|x-x_i^* \|\ge \delta'} \exp\left(-(V(x)-a(\delta')) \right)dx \\
 & \le \exp(a(\delta'))  \int_{\mathbb R^d} \exp(-V(x)) dx< \infty.
 \end{align*}
 Also, it follows that
 \begin{equation*}
 (2\pi \sigma)^{-\frac d 2}\exp \left(-\frac{a(\delta')-V(x_i^*)}{\sigma}\right) \rightarrow 0 \quad \text{as $\sigma \downarrow 0.$}
 \end{equation*}
 Thus we get
 \begin{equation}\label{eq 16}
 \lim_{\sigma \to 0} (2\pi \sigma)^{-\frac d 2} \exp\left( \frac{V(x_i^*)}{\sigma}\right) \widetilde{m}(\sigma,\delta')=0.
 \end{equation}
 Taking limit $\sigma \downarrow 0$ in \eqref{eq 14} and applying \eqref{eq 15}, \eqref{eq 16}, we get
 \begin{equation*}
\mu_{\sigma}(\{x;\|x-x_i^*\|_2 <\delta' \}) \rightarrow \frac{\left(\det \nabla^2 V(x_i^*) \right)^{-\frac 1 2}}{\sum_{j=1}^\kappa \left(\det \nabla^2 V(x_j^*) \right) ^{-\frac 1 2}} \quad \text{as $\sigma \downarrow 0.$}
 \end{equation*}
 Therefore, the proof Theorem \ref{th1} is completed.
\end{proof}


\subsection{Proof of Theorem \ref{th2}}

\begin{proof}
Note that
\begin{equation}\label{eq 1}
\mathbb P(V(X_1)>\tau)=\frac{\int_{V(x)>\tau} \exp(-V(x)/\sigma) dx}{\int_{\mathbb R^d} \exp(-V(x)/\sigma) dx}.
\end{equation}
According to $V(x)=\|x \|^2_2/2$ for any $\|x \|_2 \geq R$, then $V$ has  at least linear growth at infinity, that is, there exists  a constant $C>0$ such that for $R^{\star}$ large enough
\begin{equation*}
V(x) \ge \min_{\|y \|_2=R^{\star}} V(y)+ C(\|x \|_2 -R^{\star}) \quad \text{ for $\|x \|_2>R^{\star}$.}
\end{equation*}
We can choose sufficiently large $R^{\star}$ such that $\min_{\|y \|_2 =R^{\star}} V(y) >\tau$. Hence,
\begin{align}\label{eq 2}
\int_{V(x)\ge \tau} \exp(-V(x)/\sigma) dx &=\int_{V(x)\ge \tau, \|x\|_2 \le R^{\star}} \exp(-V(x)/\sigma) dx +\int_{V(x)\ge \tau, \|x \|_2>R^{\star}} \exp(-V(x)/\sigma) dx \notag \\
&\le \exp\left(-\frac{\tau}{\sigma} \right) \text{Vol}(B_{R^{\star}}) +\int_{V(x)\ge \tau, \|x \|_2 >R^{\star}} \exp\left(- \frac{\tau + C(\|x \|_2-R^{\star})}{\sigma} \right) dx \notag \\
& \le \exp\left(-\frac{\tau}{\sigma}\right) \left(\text{Vol}(B_{R^{\star}}) +d\text{Vol}(B_1) \int_{R^{\star}}^{\infty} r^{d-1} \exp\left\lbrace -C(r-R^{\star})\right\rbrace  dr  \right) \notag\\
& \le 2\text{Vol}(B_{R^{\star}}) \exp\left(-\frac{\tau}{\sigma}\right),
\end{align}
where Vol$(B_{R^{\star}})$ is the volume of a ball with radius $R^{\star}$. On the other hand, since $\min_x V(x) =0$, then there exists $r > 0$ such that $V(x)< \varepsilon$ when $\|x \|_2 < r$, we have
\begin{equation}\label{eq 3}
\int_{\mathbb R^d} \exp(-V(x)/\sigma) dx \ge \int_{\|x\|_2<r} \exp(-V(x)/\sigma) dx>  \exp\left(-\frac{\varepsilon}{\sigma}\right)\text{Vol}(B_r).
\end{equation}
By injection \eqref{eq 2}, \eqref{eq 3} into \eqref{eq 1}, we get
\begin{equation*}
\mathbb P(V(X_1)>\tau)\leq C_{\tau,\varepsilon,d} \exp\left(-\frac{\tau-\varepsilon}{\sigma} \right),
\end{equation*}
where
\begin{equation}\label{Const1}
C_{\tau,\varepsilon,d} := \frac{2\text{Vol}(B_{R^{\star}})}{\text{Vol}(B_r)}.
\end{equation}

Next, we will prove that the second conclusion holds in the discrete case. Recall that $s=1/K$ is the step size,  $t_k:=ks$ is the cumulative step size up to iteration $k$, and  $\{X_t\}_{t\in [0,1]}$ is the    Schr\"{o}dinger-F\"{o}llmer diffusion process \eqref{sch-equation}.
Let $\mu_{t_k}$ be the probability measure of  $Y_{t_k} $ defined by \eqref{emd}. 
Then for fixed $\tau>0$, we have
\begin{align}\label{eq 17}
\mathbb P(V(\widetilde{Y}_{t_K})>\tau)&\le  \mathbb P(V(\widetilde{Y}_{t_K})>\tau)+|\mathbb P(V(\widetilde{Y}_{t_K})>\tau) -\mathbb P(V(Y_{t_K}) > \tau)| \notag \\
&\le \mathbb P(V(Y_{t_K})>\tau)+ \|\widetilde{Y}_{t_K}-Y_{t_K} \|_{TV} \notag  \\
& \le \mathbb P(V(X_{t_K})>\tau) +\|X_{t_K}-Y_{t_K} \|_{TV}+\|\widetilde{Y}_{t_K}-Y_{t_K} \|_{TV} \notag \\
& \le  \mathbb P(V(X_{t_K})>\tau)+ \sqrt{2 \mathbb{D}_{\mathrm{KL}}(\mu_{t_K}||\mu_{\sigma})} + \sqrt{2 \mathbb{D}_{\mathrm{KL}}(\text{Law}(\widetilde Y_{t_K})||\mu_{t_K})},
\end{align}
where we use Pinsker's inequality \citep{bakry2014analysis} in the last inequality and  the second inequality holds due to the fact that letting $g(x):=\mathbbm{1}_{V(x)>\tau}$, then
\begin{align*}
|\mathbb P(V(\widetilde{Y}_{t_K}) >\tau)-\mathbb P(V(Y_{t_K}) > \tau)| &=\left|\mathbb E\left(g(\widetilde Y_{t_K}) - g( Y_{t_K}) \right) \right| \\
&= \left|\int_{\mathbb R^d} g(x)d \left(\mathbb P_{\widetilde Y_{t_k}}(x)-\mathbb P_{Y_{t_k}}(x) \right) \right|  \\
& \le |\mathbb P_{\widetilde Y_{t_k}}-\mathbb P_{Y_{t_k}} |(\mathbb R^d)
 = \|\widetilde{Y}_{t_K}-Y_{t_K} \|_{TV},
\end{align*}
where the total variation metric between two probability measures $\mu,\nu$ on $\mathbb R^d$ is defined by $\|\mu-\nu \|_{TV} :=\left|\mu-\nu \right|(\mathbb R^d)=2\sup_{A \subseteq \mathbb R^d} |\mu(A)-\nu(A)|$.

Firstly, from the first part of proof, we can get a bound for the first term on the right hand side of \eqref{eq 17}. That is, for each $\varepsilon \in (0, \tau)$, there exists a constant $C_{\tau,\varepsilon,d}$ defined in \eqref{Const1} such that
\begin{equation}\label{eq 21}
\mathbb P(V(X_{t_K})>\tau) \le C_{\tau,\varepsilon,d} \exp\left(-\frac{\tau-\varepsilon}{\sigma} \right).
\end{equation}

 Secondly, we estimate the boundness of $\mathbb{D}_{\mathrm{KL}}(\mu_{t_K}||\mu_{\sigma})$. To make use of continuous-time tools, we  construct a continuous-time interpolation for the discrete-time algorithm \eqref{emd}.  In particular, we define a stochastic process $\{Y_t \}_{t\in [0,1]}$ via SDE
 \begin{equation}\label{equation}
 dY_t=\sigma \hat b(Y_t,t)dt + \sqrt{\sigma} dB_t, \quad t \in [0,1], Y_0=0,
 \end{equation}
 with the non-anticipative drift $\hat b(Y_t,t):=\sum\limits_{k=0}^{K-1}b(Y_{t_k},t_k) \mathbbm{1}_{[t_k, t_{k+1})}(t)$. Meanwhile, because of Proposition \ref{SBP}, the process $\{X_t\}_{t\in [0,1]}$  \eqref{sch-equation}  satisfies $X_1 \sim \mu_{\sigma}$. We also denote by $\mu_1^x$ and $\nu_1^y$ the marginal distributions on $C([0,1],\mathbb R^d)$ of $(X_t,Y_t)_{t\in [0,1]}$. Thus, combining \eqref{emd}, \eqref{equation} and Lemma \ref{lemma 3.2}, we obtain
 \begin{align}\label{eq 20}
 \mathbb{D}_{\mathrm{KL}}(\mu_{t_K}||\mu_{\sigma})& \le \mathbb{D}_{\mathrm{KL}}(\nu^y_1||\mu^x_1) =\frac \sigma 2 \int_0^1 \mathbb E \left( \|b(Y_t,t)-\hat b(Y_t,t) \|_2^2\right) dt \notag \\
 &=\frac \sigma 2 \sum_{k=0}^{K-1}\int_{t_k}^{t_{k+1}} \mathbb E\left( \|b(Y_t,t)-b(Y_{t_k},t_k) \|_2^2 \right)dt \notag \\
 & \le \sigma d C_{1,\sigma}^2 \sum_{k=0}^{K-1}\int_{t_k}^{t_{k+1}} \mathbb E\left(\|Y_t-Y_{t_k} \|_2^2 +(t-t_k) \right) dt \notag \\
 &= \sigma d C_{1,\sigma}^2 \left[ \sum_{k=0}^{K-1}\int_{t_k}^{t_{k+1}} \mathbb E\left(\|b(Y_{t_k})(t-t_k)\sigma + \sqrt{\sigma} (B_t-B_{t_k}) \|_2^2 \right) dt +\frac{1}{2K} \right] \notag \\
 & \le 2\sigma d C_{1,\sigma}^2 \sum_{k=0}^{K-1}\int_{t_k}^{t_{k+1}} \mathbb E\left(\|b(Y_{t_k},t_k) \|_2^2 \sigma^2(t-t_k)^2 + \sigma d(t-t_k) \right) dt + \frac{\sigma d}{2K}  C_{1,\sigma}^2 \notag \\
 &  \le 2 d C_{1,\sigma}^2 \sum_{k=0}^{K-1}\int_{t_k}^{t_{k+1}} \left[\left(\frac{\gamma_{\sigma}}{\xi_{\sigma}} \right)^2 \sigma^3 (t-t_k)^2 +d(t-t_k)\sigma^2 \right]dt +\frac{\sigma d}{2K} C_{1,\sigma}^2 \notag \\
 & \le \frac{4d \sigma^3}{3K^2} \left( \frac{\gamma^4_{\sigma}}{\xi^4_{\sigma}} + \frac{\gamma^6_{\sigma}}{\xi^6_{\sigma}}  \right) +\frac{\sigma d(2d\sigma +1)}{K}\left( \frac{\gamma^2_{\sigma}}{\xi^2_{\sigma}} + \frac{\gamma^4_{\sigma}}{\xi^4_{\sigma}}  \right) \le \frac{d(2d+3)}{K} C^{\star}_{1,\sigma}, 
 \end{align}
where
\begin{equation*}
C^{\star}_{1,\sigma} :=\frac{\gamma^2_{\sigma}}{\xi^2_{\sigma}}+\frac{\gamma^4_{\sigma}}{\xi^4_{\sigma}}+\frac{\gamma^6_{\sigma}}{\xi^6_{\sigma}}, \quad \frac{\gamma_{\sigma}}{\xi_{\sigma}}=\left\lbrace \left(\frac{M_{2,R}}{\sigma} \right)^2 +\frac{M_{3,R}}{\sigma} \right \rbrace \exp\left(\frac{M_{1,R}-m_{1,R}}{\sigma} \right),
\end{equation*}
the second inequality holds 
due to Remark 4.1 in \cite{sfsHuang} and the fact that $(a+b)^2 \le 2(a^2 +b^2)$, the second equality holds by the continuous-time interpolation  \eqref{equation}, and the fourth inequality follows from $\|b(x,t) \|_2^2 \le \gamma_{\sigma}^2/ \xi_{\sigma}^2$.

 So it remains to estimate the relative entropy $ \mathbb{D}_{\mathrm{KL}}( \text{Law}(\widetilde Y_{t_K})||\mu_{t_K})$. Similar to the proof of the relative entropy $\mathbb{D}_{\mathrm{KL}}(\mu_{t_K}||\mu_{\sigma})$, we need to construct  a continuous-time interpolation process $\{\widetilde Y_t \}_{t \in [0,1]}$ defined by
 \begin{equation}\label{eq 18}
 d\widetilde Y_t=\sigma \hat b_m(\widetilde Y_t,t)dt + \sqrt{\sigma} dB_t, \quad t \in [0,1], \quad \widetilde Y_0=0,
 \end{equation}
 with the non-anticipative drift $\hat b_m(\widetilde Y_t,t):=\sum\limits_{k=0}^{K-1}\widetilde b_m(\widetilde Y_{t_k},t_k) \mathbbm{1}_{[t_k, t_{k+1})}(t)$, where $\widetilde b_m(\widetilde Y_k,t_k)$ is defined by \eqref{drifte1} or \eqref{drifte2}. Denote by $\nu_1^y$ and $\widetilde \nu_1^y$ the marginal distributions on $C([0,1],\mathbb R^d)$ of $(Y_t,\widetilde Y_t)_{t\in [0,1]}$. Therefore, combining \eqref{equation}, \eqref{eq 18}, Lemma \ref{lemma 3.2} and Lemma \ref{lemma6}, we get
 \begin{align}\label{eq 19}
\mathbb{D}_{\mathrm{KL}}( \text{Law}(\widetilde Y_{t_K})||\mu_{t_K}) \le \mathbb{D}_{\mathrm{KL}}(\widetilde \nu_1^y || \nu^y_1)
 &=\frac \sigma 2 \int_0^1 \mathbb E\left( \|\hat b(\widetilde Y_t,t)- \hat b_m(\widetilde Y_t,t) \|_2^2\right) dt \notag \\
 &=\frac \sigma 2 \sum_{k=0}^{K-1}\int_{t_k}^{t_{k+1}} \mathbb E\left( \| b(\widetilde Y_{t_k},t_k)- \widetilde b_m(\widetilde Y_{t_k},t_k) \|_2^2\right) dt \notag \\
  & \leq \frac{4d \sigma}{m} C^{\star}_{2,\sigma},
 \end{align}
 where
 \begin{equation*}
 C^{\star}_{2,\sigma} :=\frac{\gamma^4_{\sigma}}{\xi^4_{\sigma}} +\frac{\gamma_{\sigma}^2 \zeta^2_{\sigma}}{\xi^4_{\sigma}}, \quad\frac{\gamma_{\sigma}}{\xi_{\sigma}}=\left\lbrace \left(\frac{M_{2,R}}{\sigma} \right)^2 +\frac{M_{3,R}}{\sigma} \right \rbrace \exp\left(\frac{M_{1,R}-m_{1,R}}{\sigma} \right),
\end{equation*}
with $\zeta_{\sigma}= \exp\left(M_{1,R}/ \sigma \right)$. By injecting \eqref{eq 21}, \eqref{eq 20}, and \eqref{eq 19} into \eqref{eq 17},  we can get the desired results.
\end{proof}

\subsection{Preliminary lemmas for Theorem \ref{th3}}
First, we introduce  Lemmas \ref{lemma2}-\ref{lemma6}  preparing for the
proofs of Theorem \ref{th3}.
\begin{lemma}\label{lemma2}
Assume (\textbf{P1}) and (\textbf{P2}) hold,
then
\begin{align*}
\Ebb \left[\|X_t\|_2^2 \right]\leq 2(C_{0,\sigma}+d)\exp(2C_{0,\sigma} t).
\end{align*}
\end{lemma}
\begin{proof}
From the definition of $X_t$ in \eqref{sch-equation}, we have
$
\|X_t\|_2\leq \sigma \int_{0}^t\|b(X_u,u)\|_2\mathrm{d}u+ \sqrt{\sigma} \|B_t\|_2.
$
Then, we can get
\begin{align*}
\|X_t\|_2^2&\leq
2 \sigma^2 \left(\int_{0}^t\|b(X_u,u)\|_2\mathrm{d}u\right)^2+2 \sigma \|B_t\|_2^2\\
&\leq
2t\int_{0}^t\|b(X_u,u)\|_2^2\mathrm{d}u+2\|B_t\|_2^2\\
&\leq
2t\int_{0}^tC_{0,\sigma} \left( \|X_u\|_2^2+1\right) \mathrm{d}u+2\|B_t\|_2^2,
\end{align*}
where the first inequality holds by the inequality $(a+c)^2\leq 2a^2+2c^2$, the last inequality holds by \eqref{cond1}.
Thus,
\begin{align*}
\Ebb \left[ \|X_t\|_2^2\right] &\leq
2t\int_{0}^tC_{0,\sigma} \left( \Ebb \left[ \|X_u\|_2^2\right] +1\right) \mathrm{d}u+2\Ebb \left[ \|B_t\|_2^2\right] \\
&\leq
2C_{0,\sigma} \int_{0}^t \Ebb \left[ \|X_u\|_2^2\right] \mathrm{d}u + 2(C_{0,\sigma} +d).
\end{align*}
By the Gr\"onwall inequality, we have
\begin{align*}
\Ebb \left[ \|X_t\|_2^2\right] \leq 2(C_{0,\sigma} +d)\exp(2C_{0,\sigma} t).
\end{align*}
\end{proof}
\begin{lemma}\label{lemma3}
Assume (\textbf{P1}) and (\textbf{P2}) hold,
then for any $0\leq t_1\leq t_2\leq 1$,
\begin{align*}
\Ebb \left[ \|X_{t_2}-X_{t_1}\|_2^2 \right] \leq
2(t_2-t_1) \left\lbrace 1 + C_{0,\sigma} \exp(2\sqrt{C_{0,\sigma}}+1)\right\rbrace.
\end{align*}
\end{lemma}
\begin{proof}
Using \eqref{cond1} and the elementary inequality $2ab \le a^2+b^2$, one can derive that for any $\varepsilon >0$,
\begin{equation*}
2\left\langle x, b(x,t) \right\rangle \le \varepsilon \|x \|^2_2 +\frac{\|b(x,t) \|^2_2}{\varepsilon} \le \varepsilon \|x \|^2_2 +\frac{C_{0,\sigma}}{\varepsilon}(1+ \|x \|^2_2).
\end{equation*}
Letting $\varepsilon=\sqrt{C_{0,\sigma}}$ yields
\begin{equation}\label{eq 24}
\left\langle x, b(x,t) \right\rangle \le \sqrt{C_{0,\sigma}} (1+\|x \|^2_2).
\end{equation}
From the definition of $X_t$ in \eqref{sch-equation}, we have
\begin{equation*}
X_t=\sigma \int_{0}^t b(X_s,s) ds + \sqrt{\sigma} \int_{0}^t dB_s, \quad \forall t \in [0,1].
\end{equation*}
On the one hand, by the It\^o formula and \eqref{eq 24}, for any $ t\in [0,1]$, we have
\begin{align*}
1+\|X_t \|^2_2 &= 1+2 \sigma \int_0^t \left\langle X_s, b(X_s,s) \right\rangle ds + \int_0^t \sigma ds +2\int_0^t \left\langle X_s, \sqrt{\sigma} dB_s \right\rangle  \\
& \le 1+ 2\int_0^t \left( X^{\top}_s b(X_s,s) +\frac{1}{2} \right)ds +2\int_0^t \left\langle X_s, dB_s \right\rangle \\
& \le 1+2\alpha \int_{0}^t \left(1+ \|X_s \|^2_2 \right)ds +2\int_0^t \left\langle X_s, dB_s \right\rangle,
\end{align*}
where $\alpha :=\sqrt{C_{0,\sigma}}+ 1/2$. Furthermore, we have
\begin{equation*}
\mathbb E\left[1 +\|X_t \|^2_2 \right] \le 1+ 2\alpha \int_0^t \mathbb E\left[1+ \|X_s \|^2_2 \right] ds.
\end{equation*}
The Gr\"onwall inequality yields
\begin{equation}\label{eq 25}
\mathbb E \left[\| 1+ X_t \|_2^2 \right] \le \exp\left( 2 \alpha t \right)=\exp\left(2\sqrt{C_{0,\sigma}} +1 \right), \quad \forall t\in [0,1].
\end{equation}
On the other hand, by the elementary inequality $(a+b)^2 \le 2(a^2 + b^2)$ then we get
\begin{equation*}
\mathbb E\left[\|X_{t_2} -X_{t_1} \|^2_2 \right] \le 2 \mathbb E\left[\left(\int_{t_1}^{t_2} \sigma b(X_s,s)  ds \right)^2 \right] +2\mathbb E\left[ \left( \int_{t_1}^{t_2} \sqrt{\sigma} dB_s\right)^2 \right].
\end{equation*}
Thus, using the Cauchy-Schwarz inequality, Burkholder-Davis-Gundy inequality, \eqref{cond1} and \eqref{eq 25}, we deduce that
\begin{align*}
\mathbb E\left[\|X_{t_2} -X_{t_1} \|^2_2 \right] &\le 2\sigma^2(t_2-t_1) \int_{t_1}^{t_2} \mathbb E\left[\|b(X_s,s) \|_2^2 \right]ds +2\sigma(t_2-t_1) \\
& \le 2 C_{0,\sigma} \int_{t_1}^{t_2} \mathbb E\left[1+\|X_s \|^2_2 \right] ds +2(t_2-t_1)  \\
& \le \left\lbrace 2+ 2C_{0,\sigma} \exp(2\sqrt{C_{0,\sigma}}+1)\right\rbrace (t_2-t_1).
\end{align*}
\end{proof}
\begin{lemma}\label{lemma6}
Assume (\textbf{P1}) and (\textbf{P2}) hold, then
\begin{align*}
\underset{x\in\mathbb{R}^d,t\in [0,1]}{\sup} \mathbb{E}\left[\|b(x,t)-\tilde{b}_m(t,x)\|_2^2\right]
\leq \frac{4d}{m} C^{\star}_{2,\sigma},
\end{align*}
where
\begin{equation*}
C^{\star}_{2,\sigma} :=\frac{\gamma^4_{\sigma}}{\xi^4_{\sigma}} +\frac{\gamma_{\sigma}^2 \zeta^2_{\sigma}}{\xi^4_{\sigma}}, \quad\frac{\gamma_{\sigma}}{\xi_{\sigma}}=\left\lbrace \left(\frac{M_{2,R}}{\sigma} \right)^2 +\frac{M_{3,R}}{\sigma} \right \rbrace \exp\left(\frac{M_{1,R}-m_{1,R}}{\sigma} \right),
\end{equation*}
with $\zeta_{\sigma}= \exp\left(M_{1,R}/ \sigma \right)$.
\end{lemma}
\begin{proof}
Denote two independent sequences of independent copies of $Z\sim N(0, \bI_d)$,
that is, $\bZ=\{Z_1,\ldots,Z_m\}$ and $\bZ^{\prime}=\{Z_1^{\prime},\ldots,Z_m^{\prime}\}$.
For notation convenience, we also denote
\begin{align*}
& h :=\Ebb_{Z}\left[ \nabla \hat f_{\sigma}(x+\sqrt{(1-t)\sigma} Z)\right],~ h_m :=\frac{\sum_{i=1}^m\nabla \hat f_{\sigma}(x+\sqrt{(1-t)\sigma} Z_i)}{m},\\
& e :=\Ebb_{Z}\left[ \hat f_{\sigma}(x+\sqrt{(1-t)\sigma} Z)\right],~ e_m :=\frac{\sum_{i=1}^m \hat f_{\sigma}(x+\sqrt{(1-t)\sigma} Z_i)}{m},\\
& h_m^{\prime} :=\frac{\sum_{i=1}^m\nabla \hat f_{\sigma}(x+\sqrt{(1-t)\sigma} Z_i^{\prime})}{m},~e_m^{\prime} :=\frac{\sum_{i=1}^m \hat f_{\sigma}(x+\sqrt{(1-t)\sigma} Z_i^{\prime})}{m}.
\end{align*}
Since $h-h_m=\Ebb\left[h_m^{\prime}-h_m|\bZ\right]$, then $\|h-h_m\|^2_2\leq \Ebb\left[\|h_m^{\prime}-h_m\|^2_2|\bZ\right]$. Moreover, we have
\begin{align}\label{mm1}
\Ebb \left[\|h-h_m\|^2\right]
&\leq \Ebb \left\lbrace \Ebb\left[\|h_m^{\prime}-h_m\|^2_2|\bZ\right] \right\rbrace 
=\Ebb\left[\|h_m^{\prime}-h_m\|^2_2\right] \notag\\
&=\frac{\Ebb_{Z_1,Z_1^{\prime}}\left[ \left\|\nabla \hat f_{\sigma}(x+\sqrt{(1-t)\sigma} Z_1)-\nabla \hat  f_{\sigma}(x+\sqrt{(1-t)\sigma} Z_1^{\prime})\right\|^2_2 \right] }{m} \notag\\
&\leq \frac{\sigma(1-t)\gamma_{\sigma}^{2}}{m}\Ebb_{Z_1,Z_1^{\prime}}\left[ \left\|Z_1-Z_1^{\prime}\right\|^2_2\right] \notag\\
&\leq \frac{2d\gamma_{\sigma}^2}{m},
\end{align}
where the second inequality holds by (\textbf{P1}) and the last inequality follows from the fact that $Z_1$ and $Z'_1$ are independent standard normal distribution.
Similarly, we also have
\begin{align}\label{mm2}
\Ebb\left[ |e-e_m|^2\right] 
&\leq \Ebb \left[|e_m^{\prime}-e_m|^2\right] \notag\\
&=\frac{\Ebb_{Z_1,Z_1^{\prime}} \left[
\left|\hat f_{\sigma}(x+\sqrt{(1-t)\sigma} Z_1)-\hat f_{\sigma}(x+\sqrt{(1-t) \sigma} Z_1^{\prime})\right|^2 \right]}{m}\ \notag\\
&\leq \frac{\sigma(1-t)\gamma_{\sigma}^2}{m}\Ebb_{Z_1,Z_1^{\prime}}\left[ \left\|Z_1-Z_1^{\prime}\right\|^2_2\right] \notag\\
&\leq \frac{2d\gamma_{\sigma}^2}{m},
\end{align}
where the second inequality holds due to (\textbf{P1}). Thus, by \eqref{mm1} and \eqref{mm2}, it follows that
\begin{align}\label{mm3}
\underset{x\in \mathbb{R}^d,t\in [0,1]}{\sup}\Ebb \left[ \left\|h-h_m\right\|_2^2\right] 
\leq \frac{2d\gamma_{\sigma}^{2}}{m},
\end{align}
\begin{align}\label{mm4}
\underset{x\in \mathbb{R}^d,t\in [0,1]}{\sup}\Ebb \left[ |e-e_m|^2\right] \leq \frac{2d\gamma_{\sigma}^2}{m}.
\end{align}
Then, by (\textbf{P1}) and (\textbf{P2}),
some simple calculations yield that
\begin{align}\label{mm5}
\|b(x,t)-\tilde{b}_m(x,t)\|_2
&=\left\|\frac{h}{e}-\frac{h_m}{e_m}\right\|_2\notag\\
&\leq \frac{\|h\|_2|e_m-e|+\|h-h_m\|_2|e|}{|ee_m|}\notag\\
&\leq \frac{\gamma_{\sigma} |e_m-e|+\|h-h_m\|_2|e|}{\xi_{\sigma}^2}.
\end{align}
Recall that $\zeta_{\sigma} = \exp(M_{1,R} /\sigma)$ with
\begin{equation*} 
\quad M_{1,R} :=\max_{\|x \|_2 \le R} \left \lbrace -V(x) +\frac{\|x \|^2_2}{2} \right \rbrace,
\end{equation*}
then $\hat f_{\sigma}(x) \leq \zeta_{\sigma}$ for any $x \in B_R$.
Further, by \eqref{mm5}, it follows that for all $x\in \mathbb{R}^d$ and $t\in[0,1]$,
\begin{align}\label{mm55}
\|b(x,t)-\tilde{b}_m(x,t)\|_2^2
&\leq  2\frac{\gamma_{\sigma}^2 |e_m-e|^2+\zeta_{\sigma}^2 \|h-h_m\|_2^2}{\xi_{\sigma}^4}.
\end{align}
Then, combining \eqref{mm3}-\eqref{mm4} and \eqref{mm55}, it follows that
\begin{align*}
\underset{x\in\mathbb{R}^d,t\in [0,1]}{\sup} \mathbb{E}\left[\|b(x,t)-\tilde{b}_m(t,x)\|_2^2\right]
\leq \frac{4d}{m} C^{\star}_{2,\sigma},
\end{align*}
where
\begin{equation*}
C^{\star}_{2,\sigma} :=\frac{\gamma^4_{\sigma}}{\xi^4_{\sigma}} +\frac{\gamma_{\sigma}^2 \zeta^2_{\sigma}}{\xi^4_{\sigma}}, \quad\frac{\gamma_{\sigma}}{\xi_{\sigma}}=\left\lbrace \left(\frac{M_{2,R}}{\sigma} \right)^2 +\frac{M_{3,R}}{\sigma} \right \rbrace \exp\left(\frac{M_{1,R}-m_{1,R}}{\sigma} \right), \quad \zeta_{\sigma}= \exp\left( \frac{M_{1,R}}{\sigma} \right).
\end{equation*}
\end{proof}
\begin{lemma}\label{lemma5}
Assume (\textbf{P1}) and (\textbf{P2}) hold,
then for any $k=0,1,\ldots,K$,
\begin{align*}
\mathbb E \left[\|\widetilde{Y}_{t_{k}}\|^2_2 \right]
\leq  \frac{6\gamma_{\sigma}^2}{\xi_{\sigma}^2}+3d.
\end{align*}
\end{lemma}
\begin{proof}
Define
$\Theta_{k,t} :=\widetilde{Y}_{t_k}+ \sigma (t-t_k)\tilde{b}_m(\widetilde{Y}_{t_k},t_k)$, hence, we get $\widetilde{Y}_{t}=\Theta_{k,t}+\sqrt{\sigma}(B_t-B_{t_k})$, where $t_k \leq t \leq t_{k+1}$ with $k=0,1,\ldots,K-1$. By (\textbf{P1}) and (\textbf{P2}), it follows that for all $x \in \mathbb{R}^d$ and $t\in[0,1]$,
\begin{align}\label{eq2}
\|b(x,t)\|_2^2\leq \frac{\gamma_{\sigma}^2}{\xi_{\sigma}^2},~~ \|\tilde{b}_m(x,t)\|_2^2 \leq \frac{\gamma_{\sigma}^2}{\xi^2_{\sigma}}.
\end{align}
Then, by \eqref{eq2} and $s=1/K$, we have
\begin{align*}
\|\Theta_{k,t}\|^2_2&=\|\widetilde{Y}_{t_k}\|^2_2+\sigma^2 (t-t_k)^2\|\tilde{b}_m(\widetilde{Y}_{t_k},t_k)\|^2_2
+2\sigma (t-t_k) \left< \widetilde{Y}_{t_k}, \tilde{b}_m(\widetilde{Y}_{t_k},t_k) \right> \\
&\leq (1+s\sigma)\|\widetilde{Y}_{t_k}\|^2_2+\frac{\gamma_{\sigma}^2 (s\sigma+s^2\sigma^2)}{\xi^2_{\sigma}}.
\end{align*}
Further, we can get
\begin{align*}
\Ebb \left[ \| \left. \widetilde{Y}_{t}\|^2_2 \right| \widetilde{Y}_{t_k} \right] &=\Ebb \left[ \| \left. \Theta_{k,t}\|^2_2 \right| \widetilde{Y}_{t_k}\right]  + \sigma (t-t_k)d\\
& \leq(1+\sigma s)\|\widetilde{Y}_{t_k}\|^2_2+\frac{(s\sigma +s^2 \sigma^2)\gamma_{\sigma}^2}{\xi^2_{\sigma}}+ \sigma sd.
\end{align*}
Therefore,
\begin{align*}
\Ebb \left[ \|\widetilde{Y}_{t_{k+1}}\|^2_2\right]  \leq (1+s\sigma)\Ebb \left[ \|\widetilde{Y}_{t_k}\|^2_2\right] +\frac{(s\sigma+s^2 \sigma^2)\gamma_{\sigma}^2}{\xi^2_{\sigma}}+\sigma sd.
\end{align*}
Since $\widetilde{Y}_{t_0}=0$, then by induction, we have
\begin{align*}
\Ebb \left[ \|\widetilde{Y}_{t_{k+1}}\|^2_2 \right] 
\leq e^{(k+1)s\sigma} \left( d+ \frac{(1+s\sigma)\gamma_{\sigma}^2}{\xi^2_{\sigma}} \right) \le e \left(d+ \frac{2\gamma_{\sigma}^2}{\xi^2_{\sigma}} \right) \le \frac{6\gamma_{\sigma}^2}{\xi^2_{\sigma}}+3d.
\end{align*}
\end{proof}
\subsection{Proof of Theorem \ref{th3}}
\begin{proof}
From the definitions of $\widetilde{Y}_{t_k}$ and $X_{t_k}$, we have
\begin{align*}
\|\widetilde{Y}_{t_k}-X_{t_k}\|_2^2 &\leq \|\widetilde{Y}_{t_{k-1}}-X_{t_{k-1}}\|_2^2
+\left(\int_{t_{k-1}}^{t_k}\sigma \|b(X_u,u)-\tilde{b}_m(\widetilde{Y}_{t_{k-1}},t_{k-1})\|_2\mathrm{d} u\right)^2\\
& \qquad +2\sigma \|\widetilde{Y}_{t_{k-1}}-X_{t_{k-1}}\|_2
\left(\int_{t_{k-1}}^{t_k}\|b(X_u,u)-\tilde{b}_m(\widetilde{Y}_{t_{k-1}},t_{k-1})\|_2\mathrm{d} u\right)\\
&\leq (1+s) \|\widetilde{Y}_{t_{k-1}}-X_{t_{k-1}}\|_2^2
+(1+s)\int_{t_{k-1}}^{t_k}\|b(X_u,u)-\tilde{b}_m(\widetilde{Y}_{t_{k-1}},t_{k-1})\|_2^2\mathrm{d} u\\
&\leq (1+s) \|\widetilde{Y}_{t_{k-1}}-X_{t_{k-1}}\|_2^2
+2(1+s)\int_{t_{k-1}}^{t_k}\|b(X_u,u)-b(\widetilde{Y}_{t_{k-1}},t_{k-1})\|_2^2\mathrm{d} u\\
& \qquad  +2s(1+s)\|b(\widetilde{Y}_{t_{k-1}},t_{k-1})-\tilde{b}_m(\widetilde{Y}_{t_{k-1}},t_{k-1})\|_2^2\\
&\leq (1+s) \|\widetilde{Y}_{t_{k-1}}-X_{t_{k-1}}\|_2^2+4C_{2,\sigma}^2(1+s)\int_{t_{k-1}}^{t_k} \left[ \|X_u-\widetilde{Y}_{t_{k-1}}\|_2^2+|u-t_{k-1}|\right] \mathrm{d} u\\
& \qquad +2s(1+s)\|b(\widetilde{Y}_{t_{k-1}},t_{k-1})-\tilde{b}_m(\widetilde{Y}_{t_{k-1}},t_{k-1})\|_2^2\\
&\leq (1+s) \|\widetilde{Y}_{t_{k-1}}-X_{t_{k-1}}\|_2^2
+8C_{2,\sigma}^2(1+s)\int_{t_{k-1}}^{t_k}\|X_u-X_{t_{k-1}}\|_2^2\mathrm{d} u\\
& \qquad +8C_{2,\sigma}^2s(1+s)\|X_{t_{k-1}}-\widetilde{Y}_{t_{k-1}}\|_2^2+4C_{2,\sigma}^2 (1+s)s^2\\
& \qquad  +2s(1+s)\|b(\widetilde{Y}_{t_{k-1}},t_{k-1})-\tilde{b}_m(\widetilde{Y}_{t_{k-1}},t_{k-1})\|_2^2\\
&\leq (1+s+8C_{2,\sigma}^2(s+s^2))\|\widetilde{Y}_{t_{k-1}}-X_{t_{k-1}}\|_2^2
+8C_{2,\sigma}^2(1+s)\int_{t_{k-1}}^{t_k}\|X_u-X_{t_{k-1}}\|_2^2\mathrm{d} u\\
& \qquad +4C_{2,\sigma}^2(1+s)s^2+2s(1+s)\|b(\widetilde{Y}_{t_{k-1}},t_{k-1})-\tilde{b}_m(\widetilde{Y}_{t_{k-1}},t_{k-1})\|_2^2,
\end{align*}
where the second inequality holds due to $2ac \leq s a^2+ c^2 /s$ for $s=1/K$ and the fourth inequality holds by condition \eqref{cond3}. Then, we obtain
\begin{align*}
 &\Ebb\left[ \|\widetilde{Y}_{t_k}-X_{t_k}\|_2^2\right] \notag\\
 &\le (1+s+8C_{2,\sigma}^2(s+s^2)) \Ebb \left[ \|\widetilde{Y}_{t_{k-1}}-X_{t_{k-1}}\|_2^2\right]  +8C_{2,\sigma}^2(1+s)\int_{t_{k-1}}^{t_k}\Ebb \left[ \|X_u-X_{t_{k-1}}\|_2^2\right]  \mathrm{d} u \notag\\
&~~~ +4C_{2,\sigma}^2(s^2+s^3) +2s(1+s)\Ebb \left[ \|b(\widetilde{Y}_{t_{k-1}},t_{k-1})-\tilde{b}_m(\widetilde{Y}_{t_{k-1}},t_{k-1})\|_2^2\right] \notag\\
&\le (1+s+8C_{2,\sigma}^2(s+s^2)) \Ebb \left[ \|\widetilde{Y}_{t_{k-1}}-X_{t_{k-1}}\|_2^2\right]  +H(s,\sigma) +4C_{2,\sigma}^2(s^2+s^3) \notag \\
&~~~+2s(1+s)\Ebb\left[ \|b(\widetilde{Y}_{t_{k-1}},t_{k-1})-\tilde{b}_m(\widetilde{Y}_{t_{k-1}},t_{k-1})\|_2^2\right] \notag\\
&\le(1+s+8C_{2,\sigma}^2(s+s^2)) \Ebb\left[ \|\widetilde{Y}_{t_{k-1}}-X_{t_{k-1}}\|_2^2\right]  +H(s,\sigma)+ \frac{8s(1+s)d}{m}C^{\star}_{2,\sigma},
\end{align*}
where $H(s,\sigma) :=16 C_{2,\sigma}^2(s^2+ s^3)\left\lbrace 1 + C_{0,\sigma} \exp(2\sqrt{C_{0,\sigma}}+1)\right\rbrace +4C^{2}_{2,\sigma} (s^2 +s^3)$ follows from Lemma \ref{lemma3}, and the last inequality holds by Lemma \ref{lemma6}. Owing to $\widetilde{Y}_{t_0}=X_{t_0}=0$, we can conclude that there exists a  constant 
$C_{\sigma} >0$
such that
\begin{align*}
\Ebb \left[ \|\widetilde{Y}_{t_K}-X_{t_K}\|_2^2\right]  &\leq\frac{(1+s+8C_{2,\sigma}^2(s+s^2))^K-1}{s+8C_{2,\sigma}^2(s+s^2)}
\left[H(s,\sigma) +\frac{8s(1+s)d}{m}C^{\star}_{2,\sigma} \right]\\
& \le  C_{\sigma} \left(s C^{\star}_{3,\sigma} + \frac{16d}{m} C^{\star}_{2,\sigma} \right),
\end{align*}
where
\begin{align*}
 C_{\sigma} :=\exp(1+16C_{2,\sigma}^2), \quad C^{\star}_{3,\sigma} :=40 C^2_{2,\sigma} + 32C^2_{2,\sigma}C_{0,\sigma} \exp(2\sqrt{C_{0,\sigma}}+1).
\end{align*}
Therefore,
\begin{align*}
W^2_2(\text{Law}(\widetilde{Y}_{t_K}),\mu_{\sigma})
\leq C_{\sigma} \left ( s C^{\star}_{3,\sigma} +  \frac{16d}{m} C^{\star}_{2,\sigma} \right ).
\end{align*}
It completes the proof.
\end{proof}

%


\bibliographystyle{spbasic}  
\bibliography{ref_bib}   

\begin{thebibliography}{56}
\providecommand{\natexlab}[1]{#1}
\providecommand{\url}[1]{{#1}}
\providecommand{\urlprefix}{URL }
\expandafter\ifx\csname urlstyle\endcsname\relax
  \providecommand{\doi}[1]{DOI~\discretionary{}{}{}#1}\else
  \providecommand{\doi}{DOI~\discretionary{}{}{}\begingroup
  \urlstyle{rm}\Url}\fi
\providecommand{\eprint}[2][]{\url{#2}}

\bibitem[{Bakry et~al.(2008)Bakry, Cattiaux, and Guillin}]{bakry2008rate}
Bakry D, Cattiaux P, Guillin A (2008) Rate of convergence for ergodic
  continuous markov processes: Lyapunov versus poincar{\'e}. Journal of
  Functional Analysis 254(3):727--759

\bibitem[{Bakry et~al.(2014)Bakry, Gentil, and Ledoux}]{bakry2014analysis}
Bakry D, Gentil I, Ledoux M (2014) Analysis and geometry of Markov diffusion
  operators, vol 103. Springer

\bibitem[{Barkhagen et~al.(2021)Barkhagen, Chau, Moulines, R{\'a}sonyi,
  Sabanis, and Zhang}]{barkhagen2018stochastic}
Barkhagen M, Chau NH, Moulines {\'E}, R{\'a}sonyi M, Sabanis S, Zhang Y (2021)
  On stochastic gradient langevin dynamics with dependent data streams in the
  logconcave case. Bernoulli 27(1):1--33

\bibitem[{Bernton et~al.(2019)Bernton, Heng, Doucet, and
  Jacob}]{bernton2019schr}
Bernton E, Heng J, Doucet A, Jacob PE (2019) Schr{\"o}dinger bridge samplers.
  arXiv preprint arXiv:191213170

\bibitem[{Bou-Rabee et~al.(2020)Bou-Rabee, Eberle, and
  Zimmer}]{bou2020coupling}
Bou-Rabee N, Eberle A, Zimmer R (2020) Coupling and convergence for hamiltonian
  monte carlo. The Annals of applied probability 30(3):1209--1250

\bibitem[{Carrillo et~al.(2021)Carrillo, Jin, Li, and
  Zhu}]{carrillo2021consensus}
Carrillo JA, Jin S, Li L, Zhu Y (2021) A consensus-based global optimization
  method for high dimensional machine learning problems. ESAIM: Control,
  Optimisation and Calculus of Variations 27:S5

\bibitem[{Chau et~al.(2021)Chau, Moulines, R{\'a}sonyi, Sabanis, and
  Zhang}]{chau2019stochastic}
Chau NH, Moulines {\'E}, R{\'a}sonyi M, Sabanis S, Zhang Y (2021) On stochastic
  gradient langevin dynamics with dependent data streams: The fully non-convex
  case. SIAM Journal on Mathematics of Data Science 3(3):959--986

\bibitem[{Chen et~al.(2021)Chen, Georgiou, and Pavon}]{chen2020stochastic}
Chen Y, Georgiou TT, Pavon M (2021) Stochastic control liaisons: Richard
  sinkhorn meets gaspard monge on a schr\"{o}dinger bridge. SIAM Review
  63(2):249--313

\bibitem[{Cheng and Bartlett(2018)}]{cheng2018convergence}
Cheng X, Bartlett P (2018) Convergence of langevin mcmc in kl-divergence.
  Proceedings of Machine Learning Research, Volume 83: Algorithmic Learning
  Theory pp 186--211

\bibitem[{Cheng et~al.(2018)Cheng, Chatterji, Bartlett, and
  Jordan}]{cheng2018underdamped}
Cheng X, Chatterji NS, Bartlett PL, Jordan MI (2018) Underdamped langevin mcmc:
  A non-asymptotic analysis. In: Conference on learning theory, PMLR, pp
  300--323

\bibitem[{Chiang et~al.(1987)Chiang, Hwang, and Sheu}]{chiang1987diffusion}
Chiang TS, Hwang CR, Sheu SJ (1987) Diffusion for global optimization in
  r\^{}n. SIAM Journal on Control and Optimization 25(3):737--753

\bibitem[{Dai~Pra(1991)}]{dai1991stochastic}
Dai~Pra P (1991) A stochastic control approach to reciprocal diffusion
  processes. Applied mathematics and Optimization 23(1):313--329

\bibitem[{Dalalyan(2017{\natexlab{a}})}]{dalalyan2017further}
Dalalyan AS (2017{\natexlab{a}}) Further and stronger analogy between sampling
  and optimization: Langevin monte carlo and gradient descent. pp 678--689

\bibitem[{Dalalyan(2017{\natexlab{b}})}]{dalalyan2017theoretical}
Dalalyan AS (2017{\natexlab{b}}) Theoretical guarantees for approximate
  sampling from smooth and log-concave densities. Journal of the Royal
  Statistical Society: Series B (Statistical Methodology) 79(3):651--676

\bibitem[{Dalalyan and Karagulyan(2019)}]{dalalyan2019user-friendly}
Dalalyan AS, Karagulyan AG (2019) User-friendly guarantees for the langevin
  monte carlo with inaccurate gradient. Stochastic Processes and their
  Applications 129(12):5278--5311

\bibitem[{Dupuis and Ellis(2011)}]{dupuis2011weak}
Dupuis P, Ellis RS (2011) A weak convergence approach to the theory of large
  deviations. John Wiley \& Sons

\bibitem[{Durmus and Moulines(2016)}]{durmus2016sampling}
Durmus A, Moulines E (2016) Sampling from a strongly log-concave distribution
  with the unadjusted langevin algorithm. arXiv preprint arXiv:160501559

\bibitem[{Durmus and Moulines(2017)}]{durmus2017nonasymptotic}
Durmus A, Moulines E (2017) Nonasymptotic convergence analysis for the
  unadjusted langevin algorithm. The Annals of Applied Probability
  27(3):1551--1587

\bibitem[{Durmus and Moulines(2019)}]{durmus2016high-dimensional}
Durmus A, Moulines E (2019) High-dimensional bayesian inference via the
  unadjusted langevin algorithm. Bernoulli 25(4A):2854--2882

\bibitem[{Eldan et~al.(2020)Eldan, Lehec, and Shenfeld}]{eldan2020}
Eldan R, Lehec J, Shenfeld Y (2020) Stability of the logarithmic sobolev
  inequality via the f{\"o}llmer process. In: Annales de l'Institut Henri
  Poincar{\'e}, Probabilit{\'e}s et Statistiques, Institut Henri Poincar{\'e},
  vol~56, pp 2253--2269

\bibitem[{F{\"o}llmer(1985)}]{follmer1985}
F{\"o}llmer H (1985) An entropy approach to the time reversal of diffusion
  processes. In: Stochastic Differential Systems Filtering and Control,
  Springer, pp 156--163

\bibitem[{F{\"o}llmer(1986)}]{follmer1986}
F{\"o}llmer H (1986) Time reversal on wiener space. In: Stochastic
  processes--mathematics and physics, Springer, pp 119--129

\bibitem[{F{\"o}llmer(1988)}]{follmer1988}
F{\"o}llmer H (1988) Random fields and diffusion processes. In: {\'E}cole
  d'{\'E}t{\'e} de Probabilit{\'e}s de Saint-Flour XV--XVII, 1985--87,
  Springer, pp 101--203

\bibitem[{Garbuno-Inigo et~al.(2020)Garbuno-Inigo, Hoffmann, Li, and
  Stuart}]{garbuno2020interacting}
Garbuno-Inigo A, Hoffmann F, Li W, Stuart AM (2020) Interacting langevin
  diffusions: Gradient structure and ensemble kalman sampler. SIAM Journal on
  Applied Dynamical Systems 19(1):412--441

\bibitem[{Hale(2010)}]{hale2010asymptotic}
Hale JK (2010) Asymptotic behavior of dissipative systems. 25, American
  Mathematical Soc.

\bibitem[{Holley et~al.(1989)Holley, Kusuoka, and
  Stroock}]{holley1989asymptotics}
Holley RA, Kusuoka S, Stroock DW (1989) Asymptotics of the spectral gap with
  applications to the theory of simulated annealing. Journal of functional
  analysis 83(2):333--347

\bibitem[{Huang et~al.(2021)Huang, Jiao, Kang, Liao, Liu, and Liu}]{sfsHuang}
Huang J, Jiao Y, Kang L, Liao X, Liu J, Liu Y (2021)
  Schr\"{o}dinger-f\"{o}llmer sampler: Sampling without ergodicity. arXiv
  preprint arXiv: 210610880

\bibitem[{Hwang(1980)}]{hwang1980laplace}
Hwang CR (1980) Laplace's method revisited: weak convergence of probability
  measures. The Annals of Probability pp 1177--1182

\bibitem[{Hwang(1981)}]{hwang1981generalization}
Hwang CR (1981) A generalization of laplace's method. Proceedings of the
  American Mathematical Society 82(3):446--451

\bibitem[{Iacus and Yoshida(2018)}]{iacus2018simulation}
Iacus SM, Yoshida N (2018) Simulation and inference for stochastic processes
  with YUIMA. Springer

\bibitem[{Jamison(1975)}]{jamison1975markov}
Jamison B (1975) The markov processes of schr{\"o}dinger. Zeitschrift f{\"u}r
  Wahrscheinlichkeitstheorie und Verwandte Gebiete 32(4):323--331

\bibitem[{Jiao et~al.(2021)Jiao, Kang, Liu, and Zhou}]{sfsJiao}
Jiao Y, Kang L, Liu Y, Zhou Y (2021) Convergence analysis of
  schr\"{o}dinger-f\"{o}llmer sampler without convexity. arXiv preprint
  arXiv:210704766

\bibitem[{Landsman and Ne{\v{s}}lehov{\'a}(2008)}]{landsman2008stein}
Landsman Z, Ne{\v{s}}lehov{\'a} J (2008) Stein's lemma for elliptical random
  vectors. Journal of Multivariate Analysis 99(5):912--927

\bibitem[{Lehec(2013)}]{lehec2013representation}
Lehec J (2013) Representation formula for the entropy and functional
  inequalities. In: Annales de l'IHP Probabilit{\'e}s et statistiques, vol~49,
  pp 885--899

\bibitem[{L{\'e}onard(2014)}]{leonard2014survey}
L{\'e}onard C (2014) A survey of the schr{\"o}dinger problem and some of its
  connections with optimal transport. Discrete \& Continuous Dynamical
  Systems-A 34(4):1533

\bibitem[{Ma et~al.(2019)Ma, Chen, Jin, Flammarion, and
  Jordan}]{ma2019sampling}
Ma YA, Chen Y, Jin C, Flammarion N, Jordan MI (2019) Sampling can be faster
  than optimization. Proceedings of the National Academy of Sciences
  116(42):20881--20885

\bibitem[{M{\'a}rquez(1997)}]{marquez1997convergence}
M{\'a}rquez D (1997) Convergence rates for annealing diffusion processes. The
  Annals of Applied Probability pp 1118--1139

\bibitem[{Menz and Schlichting(2014)}]{menz2014poincare}
Menz G, Schlichting A (2014) Poincar{\'e} and logarithmic sobolev inequalities
  by decomposition of the energy landscape. Annals of Probability
  42(5):1809--1884

\bibitem[{Mou et~al.(2022)Mou, Flammarion, Wainwright, and
  Bartlett}]{mou2019improved}
Mou W, Flammarion N, Wainwright MJ, Bartlett PL (2022) Improved bounds for
  discretization of langevin diffusions: Near-optimal rates without convexity.
  Bernoulli 28(3):1577--1601

\bibitem[{Pavliotis(2014)}]{pavliotis2014stochastic}
Pavliotis GA (2014) Stochastic processes and applications: diffusion processes,
  the Fokker-Planck and Langevin equations, vol~60. Springer

\bibitem[{Peyr{\'e} and Cuturi(2019)}]{peyre2019computational}
Peyr{\'e} G, Cuturi M (2019) Computational optimal transport: With applications
  to data science. Foundations and Trends{\textregistered} in Machine Learning
  11(5-6):355--607

\bibitem[{Raginsky et~al.(2017)Raginsky, Rakhlin, and
  Telgarsky}]{raginsky2017non}
Raginsky M, Rakhlin A, Telgarsky M (2017) Non-convex learning via stochastic
  gradient langevin dynamics: a nonasymptotic analysis. In: Conference on
  Learning Theory, PMLR, pp 1674--1703

\bibitem[{Revuz and Yor(2013)}]{revuz2013continuous}
Revuz D, Yor M (2013) Continuous martingales and Brownian motion, vol 293.
  Springer Science \& Business Media

\bibitem[{Ruzayqat et~al.(2022)Ruzayqat, Beskos, Crisan, Jasra, and
  Kantas}]{ruzayqat2022unbiased}
Ruzayqat H, Beskos A, Crisan D, Jasra A, Kantas N (2022) Unbiased estimation
  using a class of diffusion processes. arXiv preprint arXiv:220303013

\bibitem[{Schr{\"o}dinger(1932)}]{schrodinger1932theorie}
Schr{\"o}dinger E (1932) Sur la th{\'e}orie relativiste de l'{\'e}lectron et
  l'interpr{\'e}tation de la m{\'e}canique quantique. In: Annales de l'institut
  Henri Poincar{\'e}, vol~2, pp 269--310

\bibitem[{Sinkhorn(1964)}]{sinkhorn1964relationship}
Sinkhorn R (1964) A relationship between arbitrary positive matrices and doubly
  stochastic matrices. The annals of mathematical statistics 35(2):876--879

\bibitem[{Stein(1972)}]{stein1972}
Stein C (1972) A bound for the error in the normal approximation to the
  distribution of a sum of dependent random variables. In: Proceedings of the
  sixth Berkeley symposium on mathematical statistics and probability, volume
  2: Probability theory, University of California Press, pp 583--602

\bibitem[{Stein(1986)}]{stein1986}
Stein C (1986) Approximate computation of expectations. IMS

\bibitem[{Tzen and Raginsky(2019)}]{tzen2019theoretical}
Tzen B, Raginsky M (2019) Theoretical guarantees for sampling and inference in
  generative models with latent diffusions. In: Conference on Learning Theory,
  PMLR, pp 3084--3114

\bibitem[{Varadhan(1966)}]{varadhan1966asymptotic}
Varadhan SS (1966) Asymptotic probabilities and differential equations.
  Communications on Pure and Applied Mathematics 19(3):261--286

\bibitem[{Vargas et~al.(2021)Vargas, Ovsianas, Fernandes, Girolami, Lawrence,
  and N{\"u}sken}]{vargas2021bayesian}
Vargas F, Ovsianas A, Fernandes D, Girolami M, Lawrence N, N{\"u}sken N (2021)
  Bayesian learning via neural schr\"odinger-f\"ollmer flows. arXiv preprint
  arXiv:211110510

\bibitem[{Wang(2009)}]{wang2009log}
Wang FY (2009) Log-sobolev inequalities: different roles of ric and hess. The
  Annals of Probability 37(4):1587--1604

\bibitem[{Wang et~al.(2021)Wang, Jiao, Xu, Wang, and Yang}]{wang2021deep}
Wang G, Jiao Y, Xu Q, Wang Y, Yang C (2021) Deep generative learning via
  schr{\"o}dinger bridge. In: International Conference on Machine Learning,
  PMLR, pp 10794--10804

\bibitem[{Wang and Li(2022)}]{wang2022accelerated}
Wang Y, Li W (2022) Accelerated information gradient flow. Journal of
  Scientific Computing 90(1):1--47

\bibitem[{Zhang and Chen(2021)}]{zhang2021path}
Zhang Q, Chen Y (2021) Path integral sampler: a stochastic control approach for
  sampling. arXiv preprint arXiv:211115141

\bibitem[{Zhang et~al.(2019)Zhang, Akyildiz, Damoulas, and
  Sabanis}]{zhang2019nonasymptotic}
Zhang Y, Akyildiz {\"O}D, Damoulas T, Sabanis S (2019) Nonasymptotic estimates
  for stochastic gradient langevin dynamics under local conditions in nonconvex
  optimization. arXiv preprint arXiv:191002008

\end{thebibliography}

\end{document}